\documentclass[11pt]{article}
\usepackage{tikz}
\usepackage{ucs}
\usepackage{amssymb}
\usepackage{amsthm}
\usepackage{amsmath}
\usepackage{latexsym}
\usepackage[cp1251]{inputenc}
\usepackage[english]{babel}
\usepackage{graphicx}
\usepackage{wrapfig}
\usepackage{caption}
\usepackage{subcaption}
\usepackage{txfonts}
\usepackage{lmodern}
\usepackage{mathrsfs}
\usepackage{changes}
\usepackage{algorithm}
\usepackage{multicol}
\usepackage{enumerate}
\usepackage{titlesec}
\usepackage{paralist}
\usepackage{mathtools}
\usepackage[T1]{fontenc}

\titleformat{\paragraph}
{\normalfont\normalsize\bfseries}{\theparagraph}{1em}{}
\titlespacing*{\paragraph}
{0pt}{3.25ex plus 1ex minus .2ex}{1.5ex plus .2ex}

\newcommand*{\myproofname}{Proof}

\usepackage{wasysym}

\usepackage{fullpage}

\def\qed{\hfill\ifhmode\unskip\nobreak\fi\qquad\ifmmode\Box\else\hfill$\Box$\fi}

\usepackage{amsthm, amssymb, amsmath}
\usepackage{verbatim, float}
\usepackage{mathrsfs}
\usepackage{graphics}
\usepackage[usenames,dvipsnames]{pstricks}
\usepackage{epsfig}
\usepackage{pst-grad} 
\usepackage{pst-plot} 
\usepackage{cases}
\usepackage{algorithmic}
\usepackage{subcaption}
\usepackage{tikz,pgfplots}

\newlength\figureheight 
\newlength\figurewidth 

\newtheorem{theorem}{Theorem}[section]

\newtheorem{lemma}[theorem]{Lemma}
\newtheorem{claim}[theorem]{Claim}
\newtheorem{proposition}[theorem]{Proposition}
\newtheorem{remark}[theorem]{Remark}
\theoremstyle{definition}
\newtheorem{defn}[theorem]{Definition}

\newtheorem{example}[theorem]{Example}
\theoremstyle{remark}

\usepackage[top=1.8cm, bottom=1.8cm, left=1.8cm, right=1.8cm]{geometry}

\title{Finding the Second-Best Candidate under the Mallows Model\thanks{Parts of the work will be presented at the International Symposium on Information Theory (ISIT) 2021, Melbourne, Australia.}}
\author{Xujun Liu and Olgica Milenkovic\\
University of Illinois, Urbana-Champaign}

\date{July 9, 2021}

\begin{document}
\maketitle

\begin{abstract}
The well-known secretary problem in sequential analysis and optimal stopping theory asks one to maximize the probability of finding the optimal candidate in a sequentially examined list under the constraint that accept/reject decisions are made in real-time. The problem has received significant interest in the mathematics community and is related to practical questions arising in online search, data streaming, daily purchase modeling and multi-arm bandit mechanisms. A version of the problem is the so-called \emph{postdoc problem,} for which the question of interest is to devise a strategy that identifies the second-best candidate with highest possible probability of success. 

We study the postdoc problem in its combinatorial form. In this setting, a permutation $\pi$ of length $N$ is sampled according to some distribution on the symmetric group $S_N$ and the elements of $\pi$ are revealed one-by-one from left to right so that at each step, one can only observe the relative orders of the elements. At each step, one must decide to either accept or reject the currently presented element and cannot recall the decision in the future. The question of interest is to find the optimal strategy for selecting the position of the second-largest value. We solve the postdoc problem for the untraditional setting where the candidates are not presented uniformly at random but rather according to permutations drawn from the Mallows distribution. The Mallows distribution assigns to each permutation $\pi \in S_N$ a weight $\theta^{c(\pi)}$, where the function $c$ counts the number of inversions in $\pi$. To identify the optimal stopping criteria for the significantly more challenging postdoc problem, we adopt a combinatorial methodology that includes new proof techniques and novel methodological extensions compared to the analysis first introduced in the setting of the secretary problem. The optimal strategies depend on the parameter $\theta$ of the Mallows distribution and can be determined exactly by solving well-defined recurrence relations.

\end{abstract}

\section{Introduction}
The secretary problem was introduced by Cayley, but the first formal description was given by Gardner~\cite{G1, G2} in 1960. In its most well-known form, the question reads as follows: $N$ individuals can be ranked from best to worst according to their qualifications, without ties. They apply for a ``secretary'' position, and are interviewed one by one, in random order. When the $i^{\text{th}}$ candidate appears, we can only compare or rank her/him relative to the $i-1$ previously seen individuals. At the time of the $i^{\text{th}}$ interview, we can hire the person presented or continue with the interview process by rejecting the current candidate. Once a rejection is made, the decision cannot be recalled. We must
select one of the $N$ individuals. What selection strategy (i.e., stopping rule) maximizes the probability of selecting the best (highest ranked) candidate? 

The first published solution was given by Lindley~\cite{L1} using direct algebraic methods while Dynkin~\cite{D1} considered the process as a Markov chain and solved the problem in a different way. The solution turns out to be surprisingly elegant and simple: reject the first $N/e$ candidates, where $e$ is the base of the natural logarithm, and then select the first candidate that outranks all previously seen candidates\footnote{The initial rejection stage is referred to the \emph{exploration stage} of stage of the process.}. This strategy ensures a probability of successfully identifying the best candidate with probability $1/e$, when $N \to \infty$.  

Both the problem formulation and solution have several practical drawbacks. If the selection policy that rejects more than $1/3$ of the candidate without regards to their qualifications is publicly known, it is hard to incentivize candidates to appear for the interview. Furthermore, the actual number of candidates appearing for an interview is usually random, with an unknown distribution. The candidates may also be presented to the evaluator in a nonuniform order (e.g. Jones~\cite{jones2020weighted}) and multiple selections or queries may be allowed (e.g.~\cite{LMM1}). 

Despite these issues, the secretary problem has attracted significant interest in the theoretical computer science and machine learning community, as modifications of the problem allow for more realistic interview settings (e.g. Szajowski~\cite{szajowski2008rank}). The prophet problem, closely related to the secretary problem but involving probabilistic models has received significant attention as illustrated in the work by Esfandiari, Hajiaghayi, Liaghat, and Monemizadeh~\cite{EHLM1} and Rubinstein~\cite{R1} (and see references therein as well). The classical paper of Kleinberg~\cite{kleinberg2005multiple} introduced a variation of the original problem in which the algorithm is allowed to choose a fixed-sized subset of candidates, and the goal is to maximize their sum (provided that the best candidates have the highest values). The work also tied this problem to online auction analysis. The interested reader is also referred to the work by Babaioff, Immorlica, Kempe, and Kleinberg~\cite{BIKK1}. A stochastic version of the secretary problem with payoff values was introduced by Bearden~\cite{bearden2006new} and used to model how traders make their selling decisions. The more recent work of Zhao, Hu, Rahimi, and King~\cite{zhao2017s} demonstrated that the Groupon data describing the behavior of users in daily deal websites can be formulated in terms of the secretary problem. The work of  Jones~\cite{J1, jones2020weighted}, Fowlkes and Jones~\cite{FJ1}, and Crews, Jones, Myers, Taalman, Urbanski, and Wilson~\cite{CJMTUW1} departed from the standard assumption that candidates are interviewed uniformly at random and proposed using the Mallows model~\cite{mallows1957non} instead. This modeling strategy is of significant practical interest as candidates are usually not interviewed blindly but based on prior reviews of their resumes, side-information provided by other institutions or other evaluation approaches. The readers are referred to the paper of Busa-Fekete, Fotakis, Sz\"or\'enyi, and Zampetakis~\cite{BFSZ1} (and the references therein) for more details on the Mallows model.

Another extension of the secretary problem is in terms of identifying the $a^{th}$-best candidate, where $a \geq 2$. The case $a=2$, for which the goal is to identify the second-best candidate, is known as the \emph{postdoc problem}, and appears to have been introduced by Dynkin in the 1980s and was further studied by Bay\'on, Ayuso, Grau, Oller-Marc\'en, and Ruiz~\cite{BAGOR1, BAGOR2}. A rationale for choosing to find and hire the second best candidate is that the best candidate may be interviewed for multiple jobs and may not accept the given offer. An optimal selection strategy similar to the one derived for the secretary problem was first proposed by Rose~\cite{rose1982postdoc} and independently analyzed by Vanderbei~\cite{vanderbei2012postdoc} using Hamilton-Jacobi-Bellman equations. An optimal strategy involves an exploration stage after which the first left-to-right second-best candidate (i.e., second-best ranked when comparing with all appeared candidates) is selected for an offer. This strategy succeeds in finding the second-best candidate with probability $1/4$, given that $N \to \infty$. 

Here, we present the first study of the postdoc problem in the (exponential) Mallows model, parametrized by $\theta>0$. Our results reveal that for $\theta > 1,$ the optimal strategy is to reject the first $k'(\theta)$ candidates and then accept the next left-to-right second-best candidate. This coincides with the optimal strategy derived in~\cite{vanderbei2012postdoc} for which $\theta=1$ and rankings are drawn uniformly at random. For $0<\theta \le 1/2,$ the optimal strategy is to reject the first $k''(\theta)$ candidates and then accept the next left-to-right \emph{best} candidate; if no selection is made before the last candidate, then  the last candidate is accepted. The most interesting optimal strategy arises for $1/2<\theta<1,$ in which case under certain constraints the optimal strategy is to reject the first $k_1(\theta)$ candidates and then accept the next left-to-right maximum, or reject the first $k_2(\theta) \ge k_1(\theta)$ candidates and then accept the next left-to-right second-maximum, whichever comes first. Although some of our proofs build upon the techniques described in~\cite{jones2020weighted}, most of the results require new combinatorial ideas and strategies that are significantly more complicated than their secretary problem counterparts. Moreover, our result implies as a special case a combinatorial proof of the classical postdoc problem ($\theta = 1$) which differs from the one presented in~\cite{rose1982postdoc} and~\cite{vanderbei2012postdoc}.

The paper is organized as follows. Section~\ref{sec:preliminaries} introduces the relevant concepts, terminology and models used throughout the paper. This section also contains a number of technical lemmas that help in establishing our main results pertaining to the optimal selection strategies described in Section~\ref{sec:strategies}. An in-depth analysis of the exploration phase length and the probability of success for the postdoc selection process under the Mallows distribution is presented in Section~\ref{sec:mallows}. Simulation results for exploration phase lengths versus $\theta$, the parameter of the Mallows distribution, are listed at the end of Section~\ref{sec:mallows}. 

\section{Preliminaries}\label{sec:preliminaries}

We assume that the sample space is the set of all permutations of $N$ elements, i.e. the symmetric group $S_N$; 
the underlying $\sigma$-algebra equals the power set of $S_N$. The best candidate is indexed by $N$, the second-best candidate by $N-1,\ldots$ while the worst candidate is indexed by $1$. We use both the term postdoc and second-best candidate to refer to the element indexed by $N-1$. It is assumed that the committee can accurately compare the candidates presented, but not the candidates unseen at the given point of the decision making process.

Unlike standard approaches for the postdoc problem, we assume that the candidates are presented according to a permutation (order) dictated by the Mallows distribution $\mathcal{M}_{\theta}$, parametrized by a real number $\theta>0$. The probability of presenting a permutation $\pi \in S_N$ to the postdoc hiring committee equals 
$$f(\pi) = \frac{\theta^{c(\pi)}}{\sum\limits_{\pi \in S_N} \theta^{c(\pi)}},$$
where $c: S_N \to \mathbb{N}$ is a permutation statistic equal to the smallest number of adjacent transpositions needed to transform $\pi$ into the identity permutation $[12\,\ldots \,N]$ (or equivalently, equal to the number of pairwise element inversions). This inversion count is known under the name \emph{Kendall} distance between the permutation $\pi$ and the identity permutation $[12\,\ldots \,N]$)\footnote{The Kendall distance is more frequently referred to as the Kendall $\tau$ distance. Since we make frequent use of the symbol $\tau$ to denote permutations and their prefixes we use the name Kendall instead of Kendall $\tau$.}. Note that the notation for a permutation in square bracket form should not be confused with the notation for a set $[a,b]=\{{a,a+1,\ldots,b\}},\, b\geq a,$ and the meaning will be clear from the context. 

For a given permutation $\pi \in S_N$ drawn according to the Mallows model, we say that a strategy \emph{wins the game} if it correctly identifies the second-best candidate when presented with $\pi$. The next definitions are based on the work of Jones~\cite{jones2020weighted}.

\begin{defn}
Given a $\pi \in S_N$, the $k^{(th)}$ prefix of $\pi$, denoted by $\pi|_k,$ is a permutation in $S_k$ that represents the relabelling of the first $k$ elements of $\pi$ according to their relative order, from smallest to largest. A {\em proper prefix} of $\pi$ is a prefix of $\pi$ with length $<|\pi|$. For example, for $\pi=[165243] \in S_6$ and $k=4$, we have $\pi|_4=[1432]$.
\end{defn}

\begin{defn}
A {\em strike set} is a list of prefixes of possibly different lengths that immediately trigger an acceptance decision for the last candidate observed. In other words, a strike set $A \subset \cup_{i=1}^{N} S_i$ corresponds to a collection of permutations $B \subset S_N$ such that for each $\sigma \in A$ with $|\sigma| = k$ we include in $B$ all permutations $\tau$ such that the $k^{(th)}$ prefix of $\tau$ equals $\sigma$; when the permutation $\tau \in S_N$ is presented, we choose to accept the $k^{(th)}$ position of $\tau$ since we see $\sigma$ when there are exactly $k$ candidates already showed up. Note that \emph{any strategy} can be represented by a strike set. During the game, if the prefix we have seen so far is not in the strike set which describes the winning strategy, then we reject the current candidate and continue.
\end{defn}

\begin{defn}
Let $\sigma \in \cup_{i=1}^{N} S_i$ and assume that the length of the permutation equals $|\sigma| = k$. We say that a $\pi \in S_N$ is $\sigma$-prefixed if $\pi|_k = \sigma$. For example, $\pi=[165243] \in S_6$ is $\sigma=[1432]$-prefixed. Given that $\pi$ is $\sigma$-prefixed, we say that $\pi$ is $\sigma$-winnable if accepting the prefix $\sigma$, i.e. if accepting the $|\sigma|^{\text{th}}$ candidate when $\sigma$ is encountered identifies the second-best candidate (i.e., \emph{wins the game}) with interview ordering $\pi$. More precisely, for $\sigma = [\sigma_1\sigma_2 \ldots \sigma_k]$, we have that $\pi$ is $\sigma$-winnable if $\pi$ is $\sigma$-prefixed and $\pi_k = N-1$.
\end{defn}

Strike sets are key to determining the optimal strategy and the largest possible probability of winning the game, as described in Theorem~\ref{winningprob}. Two other important concepts in our analysis are three conditional probabilities of winning the game based on the type of prefix encountered, defined below, and the notion of a prefix equivalent statistic (which includes the Kendall statistic). 

\begin{defn}
We say a prefix $\sigma$ is {\em eligible} if either a) it ends in a left-to-right maxima (Type I); or b) it ends in a left-to-right second maxima (Type II) or c) it has length $N$.
\end{defn}

\begin{defn}
A strike set is {\em valid} if it \\
1) Consists of prefixes that are eligible, and\\
2) It has no pair of elements such that one contains the other as prefix (i.e., the strike set is minimal), and\\
3) Every permutation $\pi \in S_N$ contains some element of the strike set as a prefix (i.e., one can always make a selection).
\end{defn}
An optimal strategy for identifying the global second-best candidate is represented by a valid strike set.

\begin{defn}
Let $\sigma$ be a permutation of length $k \le N$. We define {\em the standard denominator} $SD(\sigma)$ of $\sigma$ according to 
$$SD(\sigma)=\sum\limits_{\sigma\text{-prefixed } \pi \in S_N} \theta^{c(\pi)}.$$
\end{defn}
Throughout the remainder of the paper we also use $\bigoplus$ for the operator defined as $\frac{a}{b} \bigoplus \frac{c}{d} = \frac{a+c}{b+d}$. Using the standard denominator with the $\bigoplus$ operator allows for simplifying all pertinent explanations as one can only focus on the numerators of fractions. When a probability is written as a fraction, we view the numerator as ``the cardinality of an event'' and the denominator as ``the cardinality of the sample space'' and thus we do not cancel out their greatest common divisor to simplify the expression until the final stages of the proof. 

\begin{defn}
For a prefix $\sigma$ of length $k$ such that $1 \le k \le N$, define
\begin{align}
Q(\sigma) &= P[\text{win the game with the strategy accepting }\sigma \text{ }|\text{ }\pi \text{ is }\sigma\text{-prefixed}],\\ \notag
Q^o(\sigma) &= P[\text{win with the best strategy available after rejecting }\sigma \text{ }|\text{ }\pi \text{ is }\sigma\text{-prefixed}],\\ \notag
\bar{Q}(\sigma) &= P[\text{win with the best strategy available after rejecting}\sigma|_{k-1} \text{ }|\text{ }\pi \text{ is }\sigma\text{-prefixed}]. \notag
\end{align}
\end{defn}

Based on the previous definitions, it is clear that 
\begin{equation}\label{basicq0}
Q(\sigma) = \frac{\sum\limits_{\sigma\text{-winnable }\pi \in S_N} \theta^{c(\pi)}}{SD(\sigma)} \quad \text{ and } \quad \bar{Q}(\sigma) = \max(Q(\sigma), Q^o(\sigma)).    
\end{equation}

Intuitively, the probability $Q$ measures the chance of winning by accepting the current candidate while $Q^o$ measures the best chance to win by selecting a future candidate.

\begin{defn}\label{children}
For each $\sigma \in S_{\ell-1}$, where $\ell \le N$, we define $\sigma_j$, $1 \le j \le \ell$, to be the $\sigma$-prefixed permutation of length $\ell$ such that its last position has value $j$ after relabelling according to the first $\ell-1$ positions of $\sigma$. For example, for $\sigma = [123],$ a permutation of length $3$, we have $\sigma_1 = [2341], \sigma_2 = [1342], \sigma_3 = [1243]$ and $\sigma_4 = [1234]$.
\end{defn}

Next, let $1 \le |\sigma| = k \le N-1$. Then $Q^o(\sigma)$ represents a fraction with denominator $SD(\sigma)$ and numerator 
equal to the sum of $\theta^{c(\pi)}$ over all $\sigma$-prefixed permutations $\pi$ such that the second-best candidate (indexed by $N-1$) in $\pi$ can be selected using an optimal strategy after rejecting the $|\sigma|^{\text{th}}$ candidate. Thus,  
\begin{equation}\label{basicq0o}
Q^o(\sigma) = \bar{Q}(\sigma_1) \bigoplus \ldots \bigoplus \bar{Q}(\sigma_{k+1}) \quad \text{ and } \quad SD(\sigma) = \sum\limits_{j = 1}^{k+1} SD(\sigma_{j}).
\end{equation}

\begin{defn}
We call a prefix $\sigma$ {\em positive} if $Q(\sigma) \ge Q^o(\sigma)$ and {\em negative} otherwise. In words, a prefix $\sigma$ of length $k$ is positive if the probability of winning by accepting $\sigma_k$ is greater than or equal to the probability of winning after deciding to reject $\sigma_k$. We call a prefix $\sigma$ {\em strictly positive} if $Q(\sigma) > Q^o(\sigma)$. \end{defn}

\begin{proposition}\label{q0probs}
Let $\tau$ be any permutation of length at most $N$. The probabilities $Q^o(\tau), Q(\tau)$, and $\bar{Q}(\tau)$ can be pre-calculated using a sequential procedure.
\end{proposition}

\begin{proof}
We first observe that the prefixes of length $N$ are positive, which serves as a base case for induction on the length of a prefix. More precisely, for a permutation $\tau$ of length $N$, if $\tau(N) = N-1$ then $Q(\tau) = \bar{Q}(\tau) = 1$ and $Q^o(\tau) = 0$; if $\tau(N) \le N-2$ or $\tau(N) = N$  then $Q(\tau) = Q^o(\tau) = \bar{Q}(\tau) = 0$.

Assume that the probabilities $Q, Q^o, \bar{Q}$ for permutations of length longer than $k$, $1 \le k \le N-1$, are already known. We show that $Q^o(\tau), Q(\tau)$, and $\bar{Q}(\tau)$ can be pre-calculated, where now $\tau$ is a permutation of length $k$. By~\eqref{basicq0}, we know the value of $Q(\tau)$; the probability $Q^o(\tau)$ can be obtained from $Q^o(\tau) = \bigoplus\limits_{j = 1}^{k+1} \bar{Q}(\tau_j)$, since each $\tau_j$ has length larger than that of $\tau$; the $\bar{Q}(\omega)$ probabilities can be determined from $\bar{Q}(\tau) = \max\{Q(\tau), Q^o(\tau)\}$. 
\end{proof}

Note that this it is not the most efficient way for computing the probabilities. Lemma~\ref{relations} describes another way of computing the probabilities $Q, Q^o$ of Type I and Type II permutations of length $k,$ using the probabilities $Q, Q^o, \bar{Q}$ of Type I and Type II permutations of length $k+1$. The probabilities $Q$ equal to $0$ for prefixes that are neither Type I nor Type II. Moreover, we describe an optimal strategy in Section~\ref{sec:strategies} and show in Section~\ref{sec:mallows} how to find the maximum probability of winning through our optimal strategy using well-defined recurrence relations.

Recall that by~\eqref{basicq0}, $Q(\sigma)$ can be written as a fraction with denominator $SD(\sigma)$ and the numerator equal to the sum of $\theta^{c(\pi)}$ over all $\pi$ that are $\sigma$-winnable. Next, we show in Proposition~\ref{expansion-0} that $Q^o(\sigma)$ can be expressed in a similar manner.
\begin{proposition}\label{expansion-0}
Let $\sigma$ be a permutation of length $\ell-1$ with $\ell \le N$. There is a collection of $\sigma$-prefixed permutations $A^{\sigma}$ such that each $\mu \in A^{\sigma}$ is of length larger than $|\sigma|$ and positive, and
$$Q^o(\sigma) = \bigoplus\limits_{\mu \in A^{\sigma}} Q(\mu).$$
Moreover, the above expression is equivalent to
\begin{equation}\label{expansion-sd}
Q^o(\sigma) \cdot SD(\sigma) = \sum\limits_{\mu \in A^{\sigma}} Q(\mu) \cdot SD(\mu) \quad \text{ and } \quad SD(\sigma) = \sum\limits_{\mu \in A^{\sigma}} SD(\mu).
\end{equation}
\end{proposition}

\begin{proof}
By~\eqref{basicq0o}, we know that $Q^o(\sigma) = \bar{Q}(\sigma_1) \bigoplus \bar{Q}(\sigma_2) \bigoplus \ldots \bigoplus \bar{Q}(\sigma_{\ell})$ holds. We now describe an algorithm that establishes the proof of the proposition. 

\begin{itemize}
\item[Initialization step:] Let $A^{\sigma} = \emptyset$ and $B = \{\sigma_1, \ldots, \sigma_{\ell}\}$. 

We repeat the Main step below until the process terminates.

\item[Main step:] Check if $B = \emptyset$. If true, then stop and return the set $A^{\sigma}$; if not, then do the following: Pick a $\phi \in B$, say of length $q$ with $|\sigma| < q \le N$, check if $\phi$ is both eligible and $Q(\phi) \ge Q^o(\phi)$ holds. If true, then set $A^{\sigma} = A^{\sigma} \cup \phi$ and $B = B - \phi$; if not, then do not update $A^{\sigma}$ and let $B = B \cup \bigcup\limits_{j = 1}^{q+1} \phi_j$.
\end{itemize}

Since the permutations of length $N$ are positive, the algorithm will terminate. The Main step of the algorithm will produce a set $A^{\sigma}$ of positive eligible permutations that are also minimal. At the end of the process, $B$ is an empty set. To see this, we make the following two observations.

\textbf{Observation (i):} There is no pair of elements $\alpha,\beta \in A^{\sigma}$ such that $\alpha$ is a prefix of $\beta$, i.e., $A^{\sigma}$ contains minimal prefixes only, since otherwise the forest $T^o(\alpha)$ will not be processed by the algorithm and it will be impossible for $\beta$ to be selected for inclusion in $A^{\sigma}$.

\textbf{Observation (ii):} Since we choose a prefix only if it is positive and eligible, every prefix in $A^{\sigma}$ is positive and eligible. 

Therefore, we can write $A^{\sigma} = \{\mu_1, \ldots, \mu_r\}$ where each of the $\mu \in A^{\sigma}$ has length larger than $|\sigma|$. Furthermore, by the Main step of the algorithm,
\begin{equation}\label{expansion-0-2}
Q^o(\sigma) = Q(\mu_1) \bigoplus Q(\mu_2) \bigoplus \cdots \bigoplus Q(\mu_r).
\end{equation}
Moreover, by~\eqref{expansion-0-2}, $Q^o(\sigma)$ can be expressed as a fraction where the numerator is the sum of $\theta^{c(\pi)}$ over all $\sigma$-prefixed permutations $\pi$ whose best candidate can be captured by an optimal strategy after rejecting the $|\sigma|^{\text{th}}$ candidate, i.e., the collection of $\mu_1$-winnable, $\mu_2$-winnable, $\ldots$, $\mu_r$-winnable permutations in $S_N$. The denominator is the standard denominator, i.e., $\sum\limits_{\sigma\text{-prefixed }\pi \in S_N} \theta^{c(\pi)}$. 
\end{proof}

In Lemma~\ref{mainlemma1} and~\ref{mainlemma2}, we show that the probabilities $Q, Q^o$ of a prefix $\sigma$ only depend on its length and the relative order of the last position in $\sigma$. In Lemma~\ref{relations}, we describe the relations between the probabilities $Q, Q^o$ of a few relevant prefixes of consecutive lengths, which are used to derive Theorem~\ref{type2} and~\ref{type1} and describe the winning strategy for any prefix equivalent statistic.

\begin{defn}
Let $\bar{T}(\sigma)$ be the subtree rooted at $\sigma$, i.e., the tree comprising $\sigma$ and its children and let $T^o(\sigma) = \bar{T}(\sigma) - \sigma$ be the subforest obtained by deleting $\sigma$ from the graph $\bar{T}(\sigma)$.
\end{defn}

Since the set of all prefixes $\bigcup\limits_{i = 1}^{N} S_i$ also represents all possible positions in the game, we follow the approach suggested in~\cite{jones2020weighted} for the secretary problem and make use of {\em prefix trees} which naturally capture relations between all prefixes of a permutation. A prefix tree for the game of second-best choice with $N$ candidates is a partially ordered set defined on $\bigcup\limits_{i = 1}^{N} S_i$, where $\alpha < \beta$ if and only if $\alpha$ is a prefix of $\beta$ (see Figure~\ref{tree} for a prefix tree that represents the game with four candidates).

The following theorem establishes that there exists a valid strike set such that its corresponding strategy is optimal. The algorithm described in the proof also suggests a way to compute the optimal probability of winning the postdoc game. Note that the optimal strategy may not be unique and that each optimal strategy corresponds to a valid strike set. 
\begin{figure}
\begin{center}
  \includegraphics[scale=0.65]{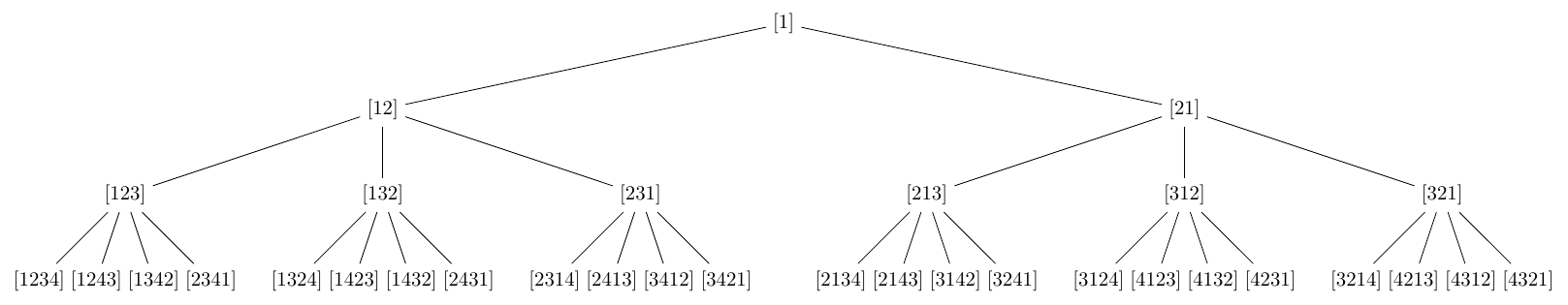}
  \caption{A prefix tree for the game of second-best choice with four candidates.}\label{tree}
\end{center}
\end{figure}

\begin{theorem}\label{winningprob}
The maximum probability of correctly identifying the second-best candidate equals
$$\bigoplus\limits_{\sigma \in A} Q(\sigma) = \frac{\sum\limits_{\sigma \in A} Q(\sigma) \cdot SD(\sigma)}{\sum\limits_{\sigma \in A} SD(\sigma)} = \frac{\sum\limits_{\sigma \in A} Q(\sigma) \cdot SD(\sigma)}{\sum\limits_{\pi \in S_N} \theta^{c(\pi)}},$$ 
where $A$ is a valid strike set with all elements positive. 
\end{theorem}

\begin{proof}
The maximum probability of winning equals $\bar{Q}([1])$, where $[1]$ is a permutation of length $1$. By Proposition~\ref{expansion-0}, the theorem holds true.
\end{proof}

Note that for a given valid strike set $A$, its corresponding strategy is to accept the candidate if the permutation (prefix) up to that point belongs to $A$. In the other direction, given a strategy, one can easily determine the corresponding strike set.

\begin{remark}
In order to identify a strategy that maximizes the probability of winning, we can actually choose to either include a permutation $\sigma$ in the set $A$ described in Theorem~\ref{winningprob} or exclude it when $Q(\sigma) = Q^o(\sigma)$. This is also the reason why an optimal strategy may not be unique.
\end{remark}



\begin{example}\label{example1}
We execute the steps of the algorithm described in Proposition~\ref{expansion-0} to find an optimal strategy and maximum probability of winning, i.e., $\bar{Q}([1])$, when $N=4$ and $\theta = 1$. We write the probabilities $(Q(\sigma), Q^o(\sigma))$ for each prefix $\sigma \in \bigcup\limits_{i = 1}^{4} S_i$ (See Figure~\ref{tree-2}). The probability of winning is $\bar{Q}([1]) = \frac{8}{24}$ which is obtained for the strike set $A = \{[12], [213], [312], [4213]\}$ (boxed in Figure~\ref{tree-2}). The strategy is: pick the first left-to-right maximum after position $1$ or the first left-to-right second-maximum after position $2$, whichever comes first; if no decisions are made before the last position, accept the corresponding candidate. Observe that there is more than one optimum strategy when $\theta = 1$; another optimum strategy is to reject the first two candidates and then accept the first second-maximum thereafter; the corresponding strike set is circled in Figure~\ref{tree-2}. 

\begin{enumerate}
\item We first compare $Q([1])$ with $Q^o([1])$. Since $Q^o([1]) = \frac{8}{24} > \frac{6}{24} = Q([1])$, let $A = \emptyset$ and $B = \{[12], [21]\}$.

\item Since $Q([12]) = \frac{4}{12} \ge \frac{4}{12} = Q^o([12])$, we have $A = \{[12]\}$ and $B = \{[21]\}$. Next we compare $Q([21])$ with $Q^o([21])$ and obtain $A = \{[12]\}$ and $B = \{[213], [312], [321]\}$.

\item We compare the prefixes in $B$. At the end of Step 3 we obtain $A = \{[12], [213], [312]\}$ and $B = \{[3214], [4213], [4312], [4321]\}$.

\item We once more compare the prefixes in $B$. The final lists are $A = \{[12], [213], [312], [4213], [3214], [4312], [4321]\}$ and $B = \emptyset$.
\end{enumerate}
\vspace{-0.1in}
\end{example}

\begin{figure*}
\begin{center}
  \includegraphics[scale=0.68]{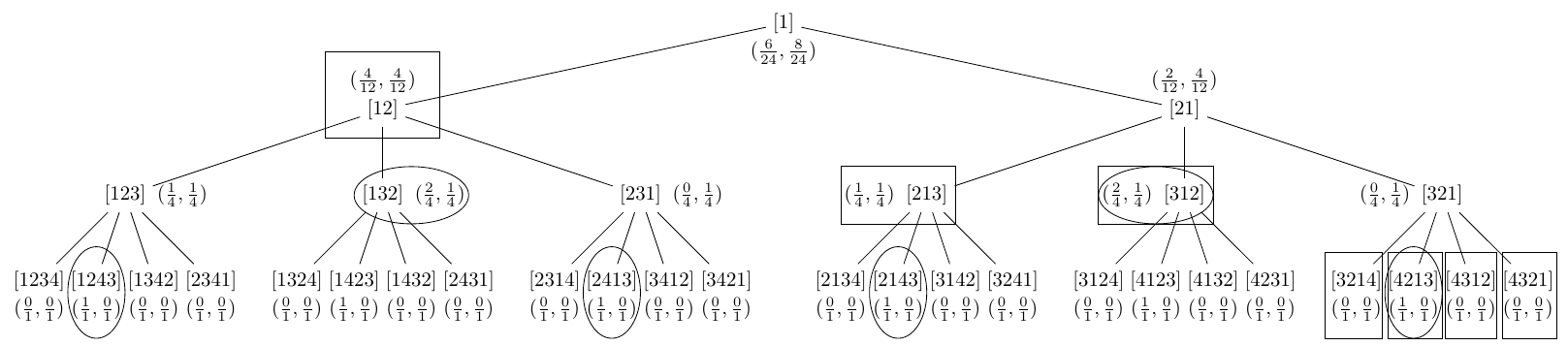}
  \caption{A prefix tree, the $Q$, and $Q^o$ probabilities for the game of second-best choice with four candidates.}\label{tree-2}
\end{center}
\vspace{-7mm}
\end{figure*}

\begin{defn}
Let $\sigma = [12 \cdots k]$ and let $g_\tau$ be a permutation operator that rearranges the elements in the permutation $\sigma$ to produce another permutation $\tau$ of length $k$. We extend the action of this operator to $\bar{T}(\sigma)$, say $\pi \in \bar{T}(\sigma)$, by similarly permuting the first $k$ entries and fixing the last $m-k$ entries of $\pi$, where $m$ is the length of $\pi \in \bar{T}(\sigma)$. 
\end{defn}
Note that the operator $g_\tau$ is a bijection from $\bar{T}(\sigma)$ to $\bar{T}(\tau)$.

\begin{example}
Let $N = 6$, $\tau = [132]$, and $\pi = [245361]$. Clearly, $\pi$ is $[123]$-prefixed and $g_{\tau} \cdot \pi =  [254361]$.
\end{example}

\begin{defn}
A statistic $c$ is {\em prefix equivalent} if it satisfies $c(\pi) - c(g_\tau \cdot \pi) = c(12 \cdots k) - c(\tau)$ for all prefixes $\tau$ and all $\pi \in \bar{T}([12 \cdots k])$, where $k$ is the length of $\tau$.  
\end{defn}

Intuitively, the condition $c(\pi) - c(g_\tau \cdot \pi) = c([12 \cdots k]) - c(\tau)$ requires the statistic $c$ to have the property that permuting the first $k$ entries does not create or remove any structure that is counted by the statistic $c$, and which lies beyond entry $k$. The condition ensures many useful properties for the probabilities $Q, Q^o, \bar{Q}$, including invariance under local changes (say, permuting the elements in a prefix). Prefix equivalence will be used intensively in the proofs of the theorems and lemmas to follow in this section. Before proceeding with the description of the more complicated results, we prove in Lemma~\ref{prefixequivalent} that the Kendall statistic is prefix equivalent.

\begin{lemma}\label{prefixequivalent}
The Kendall statistic is prefix equivalent. 
\end{lemma} 

\begin{proof}
Note that the Kendall statistic counts the number of inversions in a permutation $\pi$. Permuting the first $k$ entries will not influence any inversion involving elements in positions in $\{k+1, \ldots, N\}$ and an inversion between an entry at a position at most $k$ and another entry in a position following $k$ remains an inversion as the relative order of the two sets of elements is unchanged.
\end{proof}

In Lemma~\ref{mainlemma1}, we prove that the probabilities $Q$ of permutations only depend on the length of the permutations and the value seen at their last position; see Figure~\ref{tree-2} for an example.

\begin{lemma}\label{mainlemma1}
Let $c$ be a prefix equivalent statistic (including the Kendall statistic).

1) For all prefixes $\tau$ of length $k$, the $Q$ probabilities are preserved under the restricted bijection $g_\tau: T^o([12 \cdots k]) \to T^o(\tau)$.

2) If $\tau$ is Type I eligible, then $Q([12 \cdots k]) = Q(\tau)$.

3) If $\tau$ is Type II eligible, then $Q([12 \cdots (k-2)k(k-1)]) = Q(\tau)$.
\end{lemma}

\begin{proof}
1) Let $\tau$ be a prefix of length $k$ and let $\sigma \in T^o ([12 \cdots k])$ be of length $m$. Then, since  $c(\pi) - c(g_\tau \cdot \pi) = c([12 \cdots k]) - c(\tau)$ for all $\pi \in T^o([12 \cdots k])$, we have 
\begin{align}
Q(g_\tau \cdot \sigma) &= \frac{\sum_{g_\tau \cdot \sigma \textbf{-winnable } \pi \in S_N} \theta ^{c(\pi)}}{\sum_{g_\tau \cdot \sigma \textbf{-prefixed } \pi \in S_N}\theta ^{c(\pi)}} = \frac{\sum_{ \sigma \textbf{-winnable } \pi \in S_N} \theta ^{c(g_\tau \cdot \pi)}}{\sum_{ \sigma \textbf{-prefixed } \pi \in S_N}\theta ^{c(g_\tau \cdot \pi)}} \\ \notag
&= \frac{\sum_{ \sigma \textbf{-winnable } \pi \in S_N} \theta ^{c(\pi)}}{\sum_{ \sigma \textbf{-prefixed } \pi \in S_N}\theta ^{c(\pi)}} \cdot \frac{\theta ^ {c(\tau) - c([12 \cdots k])}}{\theta^{c(\tau) - c([12 \cdots k])}} = Q(\sigma). \notag
\end{align}

2) Let $\tau$ be a Type I eligible prefix of the same length as $\sigma$. Applying 1) of this theorem to $[12 \cdots (k-1)]$ and the prefix $\tau|_{k-1}$ of $\tau$ produces the claimed result.

3) Let $\tau$ be a Type II eligible prefix of the same length as $\sigma$. Applying 1) of this theorem to $[12 \cdots (k-1)]$ and the prefix $\tau|_{k-1}$ of $\tau$ produces the claimed result.
\end{proof}

In Lemma~\ref{mainlemma2}, we prove that the probabilities $Q^o$ only depend on the length of the underlying permutations and do not depend on the value at the last position; see Figure~\ref{tree-2} for an example.

\begin{lemma}\label{mainlemma2}
The following claims hold true for a prefix $\tau$ of length $k$ (in Claim 1), 2) and 3)). 

1) For $\sigma \in T^o([12 \cdots k])$, the probabilities $Q^o(\sigma)$ are preserved by $g_\tau$;

2) For $\sigma \in T^o([12 \cdots k])$, the probabilities $\bar{Q}(\sigma)$ are preserved by $g_\tau$;

3) One has 
$$Q^o([12 \cdots k]) = Q^o(\tau);$$ 
If $\tau$ is of Type I, then
$$\bar{Q}([12 \cdots (k-1)k]) = \bar{Q}(\tau);$$ 
If $\tau$ is of Type II, then
$$\bar{Q}([12 \cdots (k-2)k(k-1)]) = \bar{Q}(\tau).$$ 


4) If $\sigma_1$ and $\sigma_2$ are permutations with the same length, then $Q^o(\sigma_1) = Q^o(\sigma_2)$ and their $Q, \bar{Q}$ probabilities are equal, provided that they agree in the last position.
\end{lemma}

\begin{proof} The proofs of the claims in the lemma follow from straightforward algebraic manipulations.

1) If $\sigma$ has length $N$ then $Q^o(\sigma) = 0 = Q^o(g_{\tau} \cdot \sigma)$ since every permutation of length $N$ has the $Q^o$ probability equal to $0$. Thus, we may assume that $\sigma$ has length less than $N$. We prove the statement by induction on the length of $\sigma$. Assume the statement works for all $\sigma' \in T^o([12 \cdots k])$ of length at least $m+1$, where $k+1 \le m \le N-1$ and we will show the statement for $\sigma$ of length $m$.

By Proposition~\ref{expansion-0}, we know the probability $Q^o (\sigma)$ is a $\bigoplus$-sum of $Q$ probabilities, say 
$$Q^o(\sigma) = Q(r_1) \bigoplus Q(r_2) \bigoplus \cdots \bigoplus Q(r_n),$$ 
for some prefixes $r_i$ of length at least $|\sigma| + 1 =: m+1$. For each $\alpha \in T^o(\sigma)$, we know by induction hypothesis that $Q^o(\alpha) = Q^o(g_{\tau} \cdot \alpha)$. By Lemma~\ref{mainlemma1}, we know $Q(\alpha) = Q(g_{\tau} \cdot \alpha)$ as well. Thus, for the algorithm described in Proposition~\ref{expansion-0}, if we process $\sigma$ and end up obtaining the set $A = \{r_1, \ldots, r_n\}$ then when we process $g_{\tau} \cdot \sigma$ we will end up obtaining the set $A = \{g_{\tau} \cdot r_1, \ldots, g_{\tau} \cdot r_n\}$.

Therefore,
\begin{equation}\label{1st}
Q^o(g_\tau \cdot \sigma) = Q(g_\tau \cdot r_1) \bigoplus Q(g_\tau \cdot r_2) \bigoplus \cdots \bigoplus Q(g_\tau \cdot r_n)
\end{equation}
$$= \frac{\theta^{c(\tau) - c([12 \cdots k])}}{\theta^{c(\tau) - c([12 \cdots k])}} \cdot Q(r_1) \bigoplus \frac{\theta^{c(\tau) - c([12 \cdots k])}}{\theta^{c(\tau) - c([12 \cdots k])}} \cdot Q(r_2) \bigoplus \cdots \bigoplus \frac{\theta^{c(\tau) - c([12 \cdots k])}}{\theta^{c(\tau) - c([12 \cdots k])}} \cdot Q(r_n) = Q^o(\sigma).$$ 

2) The result follows from 1), Lemma~\ref{mainlemma1}, and $\bar{Q}(\sigma) = \max(Q^o(\sigma), Q(\sigma))$.

3) Let $\kappa_i$ be $[12 \cdots k]$-prefixed and of length $k+1$ such that the last position has relative value $i$, where $1 \le i \le k+1$. Similarly, let $\tau_i$ be $\tau$-prefixed and of length $k+1$ such that the last position has relative value $i$, where $1 \le i \le k+1$. By 2), $\bar{Q}(\kappa_{i}) = \bar{Q}(\tau_{i})$, for $1 \le i \le k+1$. Therefore, 
$$Q^o([12 \cdots k]) = \bigoplus\limits_{i = 1}^{k+1} \bar{Q}(\kappa_i) = \bigoplus\limits_{i = 1}^{k+1} \bar{Q}(\tau_i) = Q^o(\tau).$$ 
This establishes the correctness of the first part of the claim.

From $\bar{Q}(\sigma) = \max(Q^o(\sigma), Q(\sigma))$, Lemma~\ref{mainlemma1}, and the first part of Statement 3), we know the second and third part of 3) hold true as well.

4) $Q^o(\sigma_1) = Q^o(\sigma_2)$ follows from 3). Now, assume that $|\sigma_1| = |\sigma_2| = q$ and that the last position of $\sigma_1$ and $\sigma_2$ takes the value $x$, where $1 \le x \le q$. Let $q' = q - 1$ and let $\sigma \in T^o([12\cdots q'])$ be such that $|\sigma| = q$ and that the last position of $\sigma$ equals $x$. Applying 1) of Lemma~\ref{mainlemma1} to the $q'$th prefix of $\sigma, \sigma_1, \sigma_2$, we obtain $Q(\sigma_1) = Q(\sigma) = Q(\sigma_2)$ and thus $\bar{Q}(\sigma_1) =  \bar{Q}(\sigma_2)$.
\end{proof}

For the prefixes $\alpha = [12 \cdots (k-1)]$ and $\beta = [12 \cdots (k-3)(k-1)(k-2)]$, recall by Definition~\ref{children} we have 
$$\alpha_k = [12 \cdots k], \; \alpha_1 = [2 \cdots k1], \; \alpha_i = [1 \cdots (i-1)(i+1) \cdots ki], $$ 
$$\beta_k = [12 \cdots (k-3)(k-1)(k-2)k], \; \beta_{k-1} = [12 \cdots (k-3)k(k-2)(k-1)], \; \beta_{k-2} = [12 \cdots (k-3)k(k-1)(k-2)],$$ 
$$\beta_1 = [2 \cdots k1], \; \beta_i = [1 \cdots (i-1)(i+1) \cdots (k-2)k(k-1)i].$$ 

With a slight abuse of notation that leads to simplified expressions for probabilities of interest, we henceforth let 
$Q(\sigma)$, $Q^o(\sigma)$, and $\bar{Q}(\sigma)$ each stand for the numerators in their corresponding definitions, where the denominator is self-understood to be $\sum\limits_{\sigma\text{-prefixed }\pi \in S_N} \theta^{c(\pi)}$ and henceforth referred to as \emph{the standard denominator}. In subsequent proofs confined to this section, we omit the denominator whenever it agrees for all quantities of interest.

In~\eqref{sop'} and~\eqref{sp'} of the lemma to follow, we express the probabilities $Q^o$ and $Q$ of a Type I prefix ($\alpha$) of length $k-1$ via the probabilities $Q^o, Q,$ and $\bar{Q}$ of a Type I prefix $\alpha_k$ of length $k$ and a Type II prefix $\alpha_{k-1}$ of length $k$. Similarly, we express the probabilities $Q^o$ and $Q$ of a Type II prefix ($\beta$) of length $k-1$ via the probabilities $Q^o, Q,$ and $\bar{Q}$ of a Type I prefix $\alpha_k$ of length $k$ and a Type II prefix $\alpha_{k-1}$ of length $k$.



\begin{lemma}\label{relations}
We have
\begin{equation}\label{sop'}
Q^o(\alpha) = \bar{Q}(\alpha_k) + \bar{Q}(\alpha_{k-1}) + Q^o(\alpha_k) \cdot \sum\limits_{i=1}^{k-2} \theta^{c(\alpha_i) - c(\alpha_k)}, 
\end{equation} 
\begin{equation}\label{sp'}
Q(\alpha) = Q(\alpha_k) \cdot \sum\limits_{i=1}^{k-1} \theta^{c(\alpha_i) - c(\alpha_k)} + \theta ^{c(\alpha_{k}) - c(\alpha_{k-1})} \cdot Q(\alpha_{k-1}),
\end{equation}
\begin{equation}\label{sop''}
Q^o(\beta) = \bar{Q}(\alpha_k) \cdot \theta^{c(\beta_{k}) - c(\alpha_{k})}+\bar{Q}(\alpha_{k-1}) \cdot \theta^{c(\beta_{k-1}) - c(\alpha_{k-1})}+Q^o(\alpha_{k-1}) \cdot \sum\limits_{i=1}^{k-2} \theta^{c(\beta_i) - c(\alpha_{k-1})},
\end{equation}
\begin{equation}\label{sp''}
\text{ and } \hspace{1cm} Q(\beta) = Q(\alpha_{k-1}) \cdot \sum\limits_{i=1}^{k-2} \theta^{c(\beta_i) - c(\alpha_{k-1})}.
\end{equation}
\end{lemma}

\begin{proof}
There are $k$ children of $\alpha$ in the prefix tree, namely $\alpha_1, \ldots, \alpha_k$, and $k$ children of $\sigma''$ in the prefix tree, namely $\beta_1, \ldots, \beta_k$. The prefixes $\alpha_k$ and $\alpha_{k-1}$ are eligible so $\bar{Q}_{\alpha_k}$ and $\bar{Q}_{\alpha_{k-1}}$ are the optimal probabilities for the subtrees rooted at $\alpha_k$ and $\alpha_{k-1}$, respectively. The subtrees under each of the other $k-2$ children of $\sigma'$ are isomorphic to the subtree under $\alpha_k$ via the bijection $g_{\alpha_i}$. A $\pi \in S_N$ in $T^o(\alpha_k)$ wins if and only if $g_{\alpha_i} \cdot \pi$, which is in $T^o(\alpha_i)$, wins the game. Moreover, 
for each $\pi \in S_N$ that wins under $Q^o(\alpha_k)$, we have $\theta^{c(g_{\alpha_i} \cdot \pi)} = \theta^{c(\pi)} \cdot \theta^{c(\alpha_i) - c(\alpha_k)}$ since $c$  is a prefix equivalent statistic. 

As it is impossible for $\alpha_1, \ldots, \alpha_{k-2}$ to win, for~\eqref{sop'} we have: 
$$Q^o(\alpha) = \bar{Q}(\alpha_1)+ \ldots + \bar{Q}(\alpha_k) = Q^o (\alpha_1)+ \ldots + Q^o(\alpha_{k-2}) + \bar{Q}(\alpha_{k-1}) + \bar{Q}(\alpha_k)$$
$$ = \bar{Q}(\alpha_k) + \bar{Q}(\alpha_{k-1})+Q^o (\alpha_k) \cdot \sum\limits_{i=1}^{k-2} \theta^{c(\alpha_i) - c(\alpha_k)}.$$

For~\eqref{sp'}, note that a $\alpha$-winnable permutation can be $\alpha_i$-prefixed, $1 \le i \le k$. A $\alpha_i$-prefixed $\sigma'$-winnable permutation, where $1 \le i \le k-1$, can arise by applying $g_{\alpha_i}$, $i = 1, \ldots, k-1$, to a $\alpha_k$-winnable permutation $\pi$ (in words, $g_{\alpha_i}^{-1} \cdot \pi$, where $g_{\alpha_i}^{-1}$ is the inverse action of $g_{\alpha_i}$); this has an effect of placing the value $N-1$ (originally at position $k$ of a $\alpha_k$-winnable permutation) into position $k-1$. Moreover, a $\alpha$-winnable permutation can also be  a $\alpha_k$-prefixed permutation, which has $N$ at position $k$ and $N-1$ at position $k-1$; it can arise from a $\alpha_{k-1}$-winnable permutation which has $N-1$ at position $k$ (and so must have $N$ at position $k-1$) by applying $g_{\alpha_{k-1}}^{-1}$, the inverse action of $g_{\alpha_{k-1}}$, to convert the prefix $\alpha_{k-1}$ into the prefix $\alpha_k$. Therefore,
$$Q(\alpha) = Q(\alpha_k) \cdot \sum\limits_{i=1}^{k-1} \theta^{c(\alpha_i) - c(\alpha_k)} + \theta ^{c(\alpha_{k}) - c(\alpha_{k-1})} \cdot Q(\alpha_{k-1}).$$

For~\eqref{sop''}, similarly to the analysis performed for~\eqref{sop'}, we have 
$$Q^o(\beta) = \bar{Q}(\beta_k) + \bar{Q}(\beta_{k-1}) + Q^o (\beta_{k-2}) + \ldots + Q^o (\beta_1) = \bar{Q}(\beta_k) + \bar{Q}(\beta_{k-1})+Q^o(\alpha_{k-1}) \cdot \sum\limits_{i=1}^{k-2} \theta^{c(\beta_i) - c(\alpha_{k-1})}$$


$$ = \bar{Q}(\alpha_k) \cdot \theta^{c(\beta_{k}) - c(\alpha_{k})}+\bar{Q}(\alpha_{k-1}) \cdot \theta^{c(\beta_{k-1}) - c(\alpha_{k-1})}+Q^o(\alpha_{k-1}) \cdot \sum\limits_{i=1}^{k-2} \theta^{c(\beta_i) - c(\alpha_{k-1})}.$$

For~\eqref{sp''}, similarly to the analysis for~\eqref{sp'} and based on the fact that we know that no $\beta_{k}$- and $\beta_{k-1}$-prefixed permutation can be $\beta$-winnable (as the value in the $(k-1)^{\text{th}}$ position is already smaller than the value in two positions, i.e., $(k-2)$ and $k$) we have
$$Q(\beta) = Q(\alpha_{k-1}) \cdot \sum\limits_{i=1}^{k-2} \theta^{c(\beta_i) - c(\alpha_{k-1})}.$$
This completes the proof.
\end{proof}

\begin{theorem}\label{type2}
For any Type II prefixes $\sigma$ and $\tau$ with $|\sigma| = |\tau| - 1 = k-1$, we have that if $\tau$ is negative then $\sigma$ is negative. 
\end{theorem}

\begin{proof}
Let $\hat{\sigma} = [1 \cdots (k-3)(k-1)(k-2)]$ and $\hat{\tau} = [1 \cdots (k-2)k(k-1)]$. Suppose that $\hat{\tau}$ is negative so that $Q^o(\hat{\tau}) >  Q(\hat{\tau})$. Then by Lemma~\ref{relations} we have $$Q^o(\hat{\sigma}) = \bar{Q}(\alpha_k) \cdot \theta^{c(\beta_{k}) - c(\alpha_{k})}+\bar{Q}(\alpha_{k-1}) \cdot \theta^{c(\beta_{k-1}) - c(\alpha_{k-1})}+Q^o(\alpha_{k-1}) \cdot \sum\limits_{i=1}^{k-2} \theta^{c(\beta_i) - c(\alpha_{k-1})}$$
$$ \ge \bar{Q}(\alpha_k) \cdot \theta^{c(\beta_{k})-c(\alpha_k)} + Q^o(\alpha_{k-1}) \cdot \sum\limits_{i=1}^{k-1} \theta^{c(\beta_i) - c(\alpha_{k-1})} >  \bar{Q}(\alpha_k) \cdot \theta^{c(\beta_{k})-c(\alpha_k)} + Q(\alpha_{k-1}) \cdot \sum\limits_{i=1}^{k-1} \theta^{c(\beta_i) - c(\alpha_{k-1})} \ge Q(\hat{\sigma}). $$

The same conclusion is valid for every pair of Type II prefixes $\sigma$ and $\tau$ with $|\sigma| = |\tau| - 1 = k - 1$ since 
$$Q^o(\sigma) = Q^o(\hat{\sigma}) \cdot \theta^{c(\sigma) - c(\hat{\sigma})} > Q(\hat{\sigma}) \cdot \theta^{c(\sigma) - c(\hat{\sigma})} = Q(\sigma).$$
\end{proof}

\begin{theorem}\label{type1}
Let $\sigma$ and $\tau$ be  Type I prefixes with $|\sigma| = |\tau| - 1 = k - 1$. Let $Q^o(\alpha_{k-1}) \ge Q(\alpha_{k-1})$, where $\alpha_{k-1}$ is a Type II prefix of length $k$. Then if $\tau$ is negative then $\sigma$ is negative. 
\end{theorem}

\begin{proof}
Let $\tilde{\sigma} = [1 \cdots (k-1)]$ and $\tilde{\tau} = [1 \cdots k]$. Suppose that $\tilde{\tau}$ is negative so that $Q^o(\tilde{\tau}) >  Q(\tilde{\tau})$. Then by Lemma~\ref{relations} we have
$$Q^o(\tilde{\sigma})-Q(\tilde{\sigma}) = \bar{Q}(\alpha_k) + \bar{Q}(\alpha_{k-1}) + Q^o(\alpha_k) \cdot \sum\limits_{i=1}^{k-2} \theta^{c(\alpha_i) - c(\alpha_k)} - Q(\alpha_k) \cdot \sum\limits_{i=1}^{k-1} \theta^{c(\alpha_i) - c(\alpha_k)} - Q(\alpha_{k-1}) \cdot \theta ^{c(\alpha_{k}) - c(\alpha_{k-1})}$$

$$ = \bar{Q}(\alpha_k) + \bar{Q}(\alpha_{k-1}) - Q(\alpha_k) \cdot \theta^{c(\alpha_{k-1}) - c(\alpha_k)} - Q(\alpha_{k-1}) \cdot \theta ^{c(\alpha_{k}) - c(\alpha_{k-1})} + (Q^o(\alpha_k) - Q(\alpha_k)) \cdot \sum\limits_{i=1}^{k-2} \theta^{c(\alpha_i) - c(\alpha_k)}$$

$$>  \bar{Q}(\alpha_k) + \bar{Q}(\alpha_{k-1}) - Q(\alpha_k) \cdot \theta^{c(\alpha_{k-1}) - c(\alpha_k)} - Q(\alpha_{k-1}) \cdot \theta ^{c(\alpha_{k}) - c(\alpha_{k-1})} >  \bar{Q}(\alpha_k) - Q(\alpha_{k-1}) \cdot \theta ^{c(\alpha_{k}) - c(\alpha_{k-1})},$$

where the last inequality holds since 
$$\bar{Q}(\alpha_{k-1}) \ge Q^o(\alpha_{k-1}) = Q^o(\alpha_{k}) \cdot \theta^{c(\alpha_{k-1}) - c(\alpha_k)} > Q(\alpha_k) \cdot \theta^{c(\alpha_{k-1}) - c(\alpha_k)}.$$ 
Since $ Q^o(\alpha_{k-1}) \ge Q(\alpha_{k-1})$, we have 
$$ \bar{Q}(\alpha_k) - Q(\alpha_{k-1})\cdot \theta ^{c(\alpha_{k}) - c(\alpha_{k-1})} \ge Q^o(\alpha_k)- Q(\alpha_{k-1}) \cdot \theta ^{c(\alpha_{k}) - c(\alpha_{k-1})} \ge (Q^o(\alpha_{k-1}) - Q(\alpha_{k-1})) \cdot \theta ^{c(\alpha_{k}) - c(\alpha_{k-1})} \ge 0.$$

The same conclusion holds for every pair of Type I prefixes $\sigma$ and $\tau$ with $|\sigma| = |\tau| - 1 = k - 1$ since $Q^o(\sigma) = Q^o(\tilde{\sigma}) \cdot \theta^{c(\sigma) - c(\tilde{\sigma})} >  Q(\tilde{\sigma}) \cdot \theta^{c(\sigma) - c(\tilde{\sigma})} = Q(\sigma).$ 
\end{proof}


\section{Winning Strategies Under the Mallows Model} 
\label{sec:strategies}

Henceforth, we use $k \not \to \infty$ to denote that there exists a constant $C>0$ such that $k \le C$. We simplify our notation as follows: $Q^o_i(k)$ will henceforth denote the numerator of the probability $Q$ over the standard denominator, for type $i$ prefixes of length $k$, where $i \in [2]$. Similarly, $Q^o_i(k)$ will denote the numerator of the probability $Q^o$ over the standard denominator for type $i$ prefixes of length $k$, where $i \in [2]$. Using this notation in Lemma~\ref{relations} we have 

$$Q_1^o(k-1) = Q^o(\sigma'), Q_1(k-1)=Q(\sigma'), Q_2^o(k-1) = Q^o(\sigma''), Q_2(k-1)=Q(\sigma'').$$ 

Since $c(\tau_k') = 0$, $c(\tau_i') = k-i$, and $c(\tau_i'') = k+1-i$, the results of Lemma~\ref{relations} reduce to 
\begin{equation}\label{eq1}
Q^o_1(k-1) = \bar{Q}_1(k) + \bar{Q}_2(k) + Q_1^o(k) \cdot (\theta^{k-1} + \theta^{k-2} + \ldots + \theta^2),
\end{equation}
\begin{equation}\label{eq2}
Q_1(k-1) = Q_1(k) \cdot (\theta^{k-1} + \theta^{k-2} + \ldots + \theta) + \frac{1}{\theta} \cdot Q_2(k),
\end{equation}
\begin{equation}\label{eq3}
Q_2^o(k-1) = \bar{Q}_1(k) \cdot \theta + \bar{Q}_2(k) \cdot \theta + Q_2^o(k) \cdot (\theta^{k-1} + \theta^{k-2} + \ldots + \theta^2),
\end{equation}
\begin{equation}\label{eq4}
Q_2(k-1) = Q_2(k) \cdot (\theta^{k-1} + \theta^{k-2} + \ldots + \theta^2),
\end{equation}
where $Q_i^o(N) = 0 = Q_1(N)$, $Q_2(N) = \theta$, $Q_2^o(k) = \theta \cdot Q_1^o(k)$, and every value taken by $Q$, $Q^o$, $\bar{Q}$ is nonnegative. In this section, we will assume by default that $\theta \neq 1$ unless stated otherwise. 

\begin{defn}
Let $P_N(\theta)$ (henceforth written as $P_N$ to avoid notational clutter) be the polynomial in $\theta$ equal to $1 + \theta + \theta^2 + \cdots + \theta^{N-1}$. Furthermore, let $(P_{N})!$ be the polynomial in $\theta$ equal to $(P_{N})! = P_N\,P_{N-1} \cdots P_1.$
\end{defn}

\begin{claim}\label{q2}
One has
$$Q_2(k) = \theta^{2N-2k+1} \cdot \frac{(P_{N-2})!}{(P_{k-2})!}.$$
\end{claim}
\begin{proof}
Since~\eqref{eq4} can be written as $$Q_2(k-1) = Q_2(k) \cdot \theta^2 \cdot P_{k-2} \quad \text{ and we know that } \quad Q_2(N) = \theta,$$ we can solve the recurrence relation (details are omitted) to obtain the claimed formula.
\end{proof}
\begin{claim}\label{q1}
One has
$$Q_1(k) = \theta^{N-k-1} \cdot P_{N-k} \cdot \frac{(P_{N-2})!}{(P_{k-1})!}, \text{ where } Q_1(N) = 0.$$
\end{claim}
\begin{proof}
By Claim~\ref{q2} and~\eqref{eq2} it holds 
$$Q_1(k-1) = Q_1(k) \cdot \theta \cdot P_{k-1} + \frac{1}{\theta} \cdot Q_2(k).$$ 
Solving the recurrence (details are omitted) proves the claim.
\end{proof}

Using Claim~\ref{q2} and~\ref{q1} we arrive at
\begin{equation}\label{q2q1}
Q_2(k)/Q_1(k) = \theta^{N-k+2} \cdot \frac{P_{k-1}}{P_{N-k}} = \theta^{N-k+2} \cdot \frac{\theta^{k-1}-1}{\theta^{N-k} - 1} = \theta^2 \cdot \frac{\theta^{k-1} - 1}{1-1/\theta^{N-k}}.
\end{equation}

Since we are interested in asymptotic strategies, we assume throughout this section that $N \to \infty$. Our main results are derived in subsections~\ref{theta>1} and~\ref{theta<1}; these are followed by a discussion of general optimal strategies (without specific thresholds) in Subsection~\ref{strategymallows}. The precise optimal strategies (with specific thresholds) and the optimal probabilities are presented in Section~\ref{sec:mallows}.

\subsection{The Case $\theta>1$ (and $\theta=1$)}\label{theta>1}

\begin{theorem}\label{thetalarge}
Let $2 \le k<N$. If $Q^o_1(k)>Q_1(k)$, then $Q^o_1(k-1)>Q_1(k-1)$.
\end{theorem}
\begin{proof}
Let $f(k):= \bar{Q}_1(k) + \bar{Q}_2(k) - \theta \cdot Q_1(k) - \frac{1}{\theta} \cdot Q_2(k)$.

\textbf{Case A: $k \to \infty$.} Since $1-1/\theta \le 1-1/\theta^{N-k} < 1$, it holds that$~\eqref{q2q1} \to \infty$. Moreover, since $Q^o_1(k)>Q_1(k)$, $\theta > 1$ is a constant, $\eqref{q2q1} \to \infty$, and~\eqref{eq1} and~\eqref{eq2} hold true, we have
$$Q_1^o(k-1) - Q_1(k-1) = f(k) + (Q_1^o(k) - Q_1(k)) \cdot (\theta^2 + \ldots + \theta^{k-1})> f(k) \ge (1-\theta) \cdot Q_1(k) + (1-1/\theta) \cdot Q_2(k) \ge 0;$$
the second inequality holds since $\bar{Q}_1(k) \ge Q_1(k)$ and $\bar{Q}_2(k) \ge Q_2(k)$.

\textbf{Case B: $k \not \to \infty$.} Then $\eqref{q2q1} \to \theta^2 \cdot (\theta^{k-1} - 1)$, and by $\theta>1$ and $k \ge 2$ we have 
\begin{equation}\label{theta}
\theta \cdot (\theta^{k-1} - 1) > 1.
\end{equation}

\textbf{Case B.1:} $Q_2(k) \le \theta \cdot Q_1(k).$ Since $Q_1^o(k) > Q_1(k)$, we have $Q_2^o(k) = \theta \cdot Q_1^o(k) > \theta \cdot Q_1(k)$. Thus $$Q_1^o(k-1) - Q_1(k-1) = f(k) + (Q_1^o(k) - Q_1(k)) \cdot (\theta^2 + \ldots + \theta^{k-1}) > f(k) $$
$$ \ge Q_1(k) + Q_2^o(k) - \theta \cdot Q_1(k) - \frac{1}{\theta} \cdot Q_2(k)$$
$$> Q_1(k) + \theta \cdot Q_1(k) - \theta \cdot Q_1(k) - \frac{1}{\theta} \cdot Q_2(k) = Q_1(k) - \frac{1}{\theta} \cdot Q_2(k) \ge 0;$$
the second inequality holds since $\bar{Q}_1(k) \ge Q_1(k)$ and $\bar{Q}_2(k) \ge Q_2^o(k)$, while the third and fourth inequality follow from the first line of Case B.1.

\textbf{Case B.2:} $Q_2(k) > \theta \cdot Q_1(k)$. Therefore, by $Q_2(k)/Q_1(k) \to \theta^2 \cdot (\theta^{k-1} - 1)$, and from~\eqref{theta} and $\theta > 1$,
$$Q_1^o(k-1) - Q_1(k-1) > f(k) \ge Q_1(k) + Q_2(k) - \theta \cdot Q_1(k) - \frac{1}{\theta} \cdot Q_2(k)$$
$$= \text{ } Q_1(k) + (\theta^{k-1}-1) \cdot \theta^2 \cdot Q_1(k) - \theta \cdot Q_1(k) - \theta \cdot (\theta^{k-1} - 1) \cdot Q_1(k) $$
$$
= (1+\theta^{k+1} - \theta^2 - \theta^k) \cdot Q_1(k) = (\theta^k-\theta-1) \cdot (\theta-1) \cdot Q_1(k)> 0.
$$
\end{proof}

\begin{remark}\label{theta=1}
For $\theta = 1$, by taking the difference of~\eqref{eq1} and~\eqref{eq2}, we have 
$$Q_1^o(k-1) - Q_1(k-1) = \bar{Q}_1(k) - Q_1(k) + \bar{Q}_2(k) - Q_2(k) + (Q_1^o(k) - Q_1(k)) \cdot (k-2).$$ 
Since $Q_1(N) = Q_1^o(N) = 0$ and $Q_2(N) = \theta > 0 = Q_2^o(N)$, $Q_1^o(k) - Q_1(k)$ remains zero until the inequality $Q_2(k)<Q_2^o(k)$ starts to hold. More precisely, by Theorem~\ref{type1}, if $k_2$ is the largest index such that $Q_2(k_2)<Q_2^o(k_2)$, then the largest index $k_1$ such that $Q_1(k_1)<Q_1^o(k_1)$ equals $k_2-1$.
\end{remark}

\subsection{The case $\theta<1$}\label{theta<1}
\subsubsection{The subcase $0< \theta < 1/2$}\label{thetasmall}

Since the Type II prefixes of length at most $N-1$ are negative, we only need to consider Type I prefixes. By Theorem~\ref{type1}, there exists a threshold $k_1$ for negative Type I prefixes and positive Type I prefixes. 

\subsubsection{The subcase $1/2<\theta<1$}\label{thetamiddle}

\begin{theorem}\label{thetamiddle1}
Let $\sigma$ and $\tau$ be Type I prefixes with $|\sigma| = |\tau| - 1 = k - 1$. Let $N-k \to \infty$. Then if $\tau$ is negative then $\sigma$ is negative.
\end{theorem}

\begin{proof}
Let $N \to \infty$ and $f(k):= \bar{Q}_1(k) + \bar{Q}_2(k) - \theta \cdot Q_1(k) - \frac{1}{\theta} \cdot Q_2(k)$. By Claim~\ref{q2} and Claim~\ref{q1},
\begin{equation}\label{q1q2}
\frac{Q_1(k)}{Q_2(k)} = \frac{1}{\theta^2} \cdot \frac{1}{\theta^{N-k}} \cdot \frac{P_{N-k}}{P_{k-1}} = \frac{1}{\theta^2} \cdot \frac{1}{\theta^{N-k}} \cdot \frac{1-\theta^{N-k}}{1-\theta^{k-1}} \to \infty.
\end{equation}
Moreover, since $\tau$ is Type I negative, one has $Q^o_1(k)>Q_1(k)$; by noting that $1/2 < \theta < 1$ is a constant, that $\bar{Q}_1(k) \ge Q_1(k)$, $\bar{Q}_2(k) \ge Q_2(k)$, and from~\eqref{q1q2}, we obtain 
$$Q_1^o(k-1) - Q_1(k-1) = f(k) + (Q_1^o(k) - Q_1(k)) \cdot (\theta^2 + \ldots + \theta^{k-1})> f(k) \ge (1-\theta) \cdot Q_1(k) + (1-1/\theta) \cdot Q_2(k)>0.$$
\end{proof}

\begin{theorem}\label{thetamiddle2}
Let $\sigma'$ and $\sigma''$ be a Type I prefix and a Type II prefix of length $k$, respectively. Let $1/2<\theta<1$ and $N-k \not \to \infty$. For $k<N$, if $\sigma''$ is (strictly) positive then $\sigma'$ is (strictly) positive.
\end{theorem}
\begin{proof}
The proof is postponed to the Appendix (Section~\ref{appendix}).
\end{proof}
 
Since $Q_2(N) = \theta > 0 = Q_2^o(N)$, Type II prefixes of length $N$ are strictly positive. By Theorem~\ref{type2}, there is a threshold $k_2(\theta)$ such that all Type II prefixes of length at most $k_2(\theta)$ are negative and all Type II prefixes of length at least $k_2(\theta)+1$ are positive. By Theorem~\ref{thetamiddle1},~\ref{thetamiddle2}, and~\ref{type1}, we know that there is another threshold $k_1(\theta) \le k_2(\theta)$ such that all Type I prefixes of length at most $k_1(\theta)$ are negative and all Type I prefixes of length at least $k_1(\theta)+1$ are positive.
 
\begin{remark}\label{theta=1/2}
We separately discuss the case $\theta = \frac{1}{2}$. By~\eqref{eq1'},~\eqref{eq2'},~\eqref{eq3'}, and~\eqref{eq4'}, we have that when $N \to \infty$ the standard numerators satisfy
$$Q_1^o(N) = 0, Q_1(N) = 0, Q_2^o(N) = 0, Q_2(N) = 1/2;$$
$$Q_1^o(N-1)= 1/2, Q_1(N-1) = 1, Q_2^o(N-1) = 1/4, Q_2(N-1) = 1/4;$$
$$Q_1^o(N-2) = 3/2, Q_1(N-2) =  3/2, Q_2^o(N-2) = 3/4, Q_2(N-2) =  1/8;$$
$$Q_1^o(N-3) = 3, Q_1(N-3) =  7/4, Q_2^o(N-3) = 3/2, Q_2(N-3) =  1/16.$$

Therefore, by Theorem~\ref{type2} we know that Type II prefixes of length at most $N-2$ are negative; that we can be indifferent (i.e., either reject or accept) to Type II prefixes of length $N-1$; that Type II prefixes of length $N$ are strictly positive; that by Theorem~\ref{type1} we know that Type I prefixes of length at most $N-3$ are negative; that we can be indifferent (i.e., either reject or accept) to Type I prefixes of length $N-2$; that Type I prefixes of length $N-1$ are strictly positive; and that we can be indifferent to Type I prefixes of length $N$.
\end{remark}


\subsection{Optimal strategies}\label{strategymallows}
Let $\pi' = [12 \cdots (N-1)]$ and $\pi'' = [12 \cdots (N-3)(N-1)(N-2)]$.  Recall by Definition~\ref{children}, we can define $\pi'_i$ and $\pi''_i$, where $1 \le i \le N$. Furthermore, unlike in Section~\ref{sec:preliminaries} and previous subsections in Section~\ref{sec:strategies}, we now use $Q, Q^o, \bar{Q}$ to denote the original probabilities, and not only their numerators corresponding to the standard denominator. We describe an optimal strategy for each $\theta > 0$ and $N \to \infty$.


Next, note that $c(\pi'_N) = 0$, $c(\pi'_{N-1}) = 1$, $c(\pi'_{i}) = N-i$ for $i \in \{1, \ldots, N-2\}$.  

We compare
\begin{equation}\label{qpi'}
Q(\pi') = \frac{\theta^{c(\pi'_{N})}}{\theta^{c(\pi'_1)} + \theta^{c(\pi'_2)} + \ldots + \theta^{c(\pi'_N)}} = \frac{1}{\theta^{N-1} + \theta^{N-2} + \ldots + 1}
\end{equation} 
\begin{equation}\label{qopi'}
\quad \text{ and } \quad Q^o(\pi') = \frac{\theta^{c(\pi'_{N-1})}}{\theta^{c(\pi'_1)} + \theta^{c(\pi'_2)} + \ldots + \theta^{c(\pi'_N)}} = \frac{\theta}{\theta^{N-1} + \theta^{N-2} + \ldots + 1},
\end{equation}
as well as
\begin{equation}\label{qpi''}
Q(\pi'') = \frac{\theta^{c(\pi''_1)} + \theta^{c(\pi''_2)} + \ldots + \theta^{c(\pi''_{N-2})}}{\theta^{c(\pi''_1)} + \theta^{c(\pi''_2)} + \ldots + \theta^{c(\pi''_N)}} = \frac{\theta^N + \theta^{N-1} + \ldots + \theta^3}{\theta^N + \theta^{N-1} + \ldots + \theta^3+\theta^2+\theta} 
\end{equation}
\begin{equation}\label{qopi''}
\quad \text{ and } \quad Q^o(\pi'') = \frac{\theta^{c(\pi''_{N-1})}}{\theta^{c(\pi''_1)} + \theta^{c(\pi''_2)} + \ldots + \theta^{c(\pi''_N)}} = \frac{\theta^2}{\theta^N + \theta^{N-1} + \ldots + \theta^3+\theta^2+\theta}.
\end{equation}

Results from subsections~\ref{theta>1} and~\ref{theta<1} allow us to determine the winning strategies based on the probabilities $Q^o(\pi')$, $Q(\pi')$, $Q^o(\pi'')$, and $Q(\pi'')$. Note that the results in Theorems~\ref{type1},~\ref{type2},~\ref{thetalarge},~\ref{thetamiddle1}, and~\ref{thetamiddle2} still hold for the probabilities $Q^o$, $Q$, as the prefixes are of the same length and the standard denominator is positive.

\textbf{Case 1:} $\theta > 1$. By~\eqref{qpi'},~\eqref{qopi'},~\eqref{qpi''}, and~\eqref{qopi''}, we have that $Q(\pi')< Q^o(\pi')$ and $Q(\pi'') > Q^o(\pi'')$. All the Type I prefixes of length at most $N-1$ are negative by Theorem~\ref{thetalarge} and furthermore $Q(\pi'_{N}) = Q^o(\pi'_{N}) = 0$. Thus, we only need to consider Type II prefixes. By Theorem~\ref{type2}, the goal is to solve for $k_2(\theta)$ such that all the Type II prefixes of length $\leq k_2(\theta)$ are negative, and all the Type II prefixes of length greater than $k_2$ are positive. Thus, the optimal strategy in this case is to reject the first $k_2$ candidates (where $0 \le k_2 \le N-1$) and then accept the next left-to-right second-maximum thereafter. The precise parameter values are described in Section~\ref{a}.

\textbf{Case 2:} $\theta = 1$. By~\eqref{qpi'},~\eqref{qopi'},~\eqref{qpi''}, and~\eqref{qopi''}, we have $Q(\pi') = Q^o(\pi')$ and $Q(\pi'') > Q^o(\pi'')$. By Remark~\ref{theta=1} and Theorem~\ref{type1}, we need to determine a $k_1$ and a $k_2$, such that $k_1 = k_2 - 1$. The optimal strategy is to reject the first $k_1$ candidates, then be indifferent (either accept or reject) to any left-to-right maximum thereafter, and reject the $k_1+1$th candidate if it is not a left-to-right maximum and then accept the next left-to-right second-maximum. The precise parameter values are described in Remark~\ref{theta=1solve} at the end of Section~\ref{a}.

\textbf{Case 3:} $0 < \theta < 1$. By~\eqref{qpi'} and~\eqref{qopi'}, we have $Q(\pi') > Q^o(\pi')$. By~\eqref{qpi''} and~\eqref{qopi''}, we only need to compare the numerators of $Q(\pi'') $ and $Q^o(\pi'')$, i.e., $\theta^N + \theta^{N-1} + \ldots + \theta^3$ and $\theta^2$.

\textbf{Case 3.1:} $0 < \theta < \frac{1}{2}$. Then $Q(\pi'') < Q^o(\pi'')$ and $Q(\pi_{N-1}') > Q^o(\pi_{N-1}')$. By Theorem~\ref{type2}, all Type II prefixes of length at most $N-1$ are negative (even though Type II prefixes of length $N$ are positive). The best strategy is to only consider Type I prefixes and accept the last candidate no matter what, i.e., the best strategy is to reject the first $k_1$ candidates and then accept the next left-to-right maximum. If no selection is made before the last candidate, the latter is accepted. The precise parameter settings are stated Section~\ref{c1}.

\textbf{Case 3.2:} $\theta = \frac{1}{2}$. By Remark~\ref{theta=1/2}, the optimal strategy is to 1) reject all but the last three candidates; 2) if the third-last candidate is a left-to-right maximum, we can either accept or reject him/her; otherwise we reject this candidate; 3) if the second-last candidate is a left-to-right maximum, then we accept him/her. Or, if the second-last candidate is a left-to-right second-maximum, then we can either decide to accept or reject; otherwise we reject this candidate; 4) if the last candidate is a left-to-right second-maximum, we accept him/her; otherwise we can either accept or reject the candidate.


\textbf{Case 3.3:} $\frac{1}{2} < \theta < 1$. When $N \to \infty$ we have $Q(\pi'') > Q^o(\pi'')$. By Theorem~\ref{type2}, there is a $0\le k_2(\theta) \le N-2$ such that every Type II prefix of length at most $k_2(\theta)$ is negative and every Type II prefix of length longer than $k_2(\theta)$ is positive. We then have two cases to consider. We show that Case 3.3.1 is impossible and then focus on Case 3.3.2.

\textbf{Case 3.3.1:} $N-k_2 \not \to \infty$. Since all Type II prefixes with length $\ell$ such that $N-\ell \not \to \infty$ are positive, we know by Theorem~\ref{thetamiddle2} that every Type I prefix of length $\ell$ with $N-\ell \not \to \infty$ is also positive. Suppose now that $N- \ell \to \infty$. By Theorem~\ref{thetamiddle1}, there exists a $k_1(\theta) \ge 0$ with $N-k_1(\theta) \to \infty$ such that every Type I prefix of length at most $k_1(\theta)$ is negative and every Type I prefix of length longer than $k_1(\theta)$ is positive. In this case, the optimal strategy is a $(k_1, k_2)$-strategy, where $k_1 \le k_2$, or a $(k_2, k_1)$-strategy, where $k_2 \le k_1$. In other words, for a fixed $\frac{1}{2}<\theta<1$, there exists a pair of numbers $k_1, k_2$ such that the optimal strategy under the assumption for this case is either (1) reject the first $k_1$ candidates and then accept the next left-to-right maximum thereafter or reject the first $k_2 \ge k_1$ candidates and then accept the next left-to-right second-maximum thereafter, whichever appears first; or, (2) reject the first $k_2$ candidates and then accept the next left-to-right second-maximum thereafter or reject the first $k_1 \ge k_2$ candidates and then accept the next left-to-right maximum thereafter, whichever appears first. However, we show in Section~\ref{c3} that the optimal strategy among all $(k_1, k_2)$-strategies and $(k_2, k_1)$-strategies always arises when $N-k_1(\theta) \not \to \infty$ and $N-k_2(\theta) \not \to \infty$, which implies that Case 3.3.1 is impossible.

\textbf{Case 3.3.2:} $N - k_2 \to \infty$. Then by Theorem~\ref{thetamiddle2}, every Type I prefix of length longer than $k_2(\theta)$ is positive. Furthermore, by Theorem~\ref{type1}, since every Type II prefix of length at most $k_2(\theta)$ is negative, we conclude that there exists a $0 \le k_1(\theta) \le k_2(\theta) \le N-2$ such that every Type I prefix of length at most $k_1(\theta)$ is negative and every Type I prefix of length larger than $k_1(\theta)$ is positive. Therefore, the optimal strategy is the $(k_1(\theta), k_2(\theta))$-strategy, i.e., we reject the first $k_1(\theta)$ candidates and then accept the next left-to-right maximum thereafter or reject the first $k_2(\theta) \ge k_1(\theta)$ candidates and then accept the next left-to-right second-maximum thereafter, whichever appears first. The precise parameter settings are described in Section~\ref{c3}.

\section{Precise Parameter Settings for the Mallows Model}\label{sec:mallows}

The following result is well-known and also proved in~\cite{jones2020weighted}.

\begin{lemma} [Lemma 6.2 in~\cite{jones2020weighted},~\cite{mallows1957non}]
We have $$(P_{N})! =  \sum\limits_{\pi \in S_N} \theta^{\#\text{inversions in }\pi}.$$
\end{lemma}

For the set $[1, n+m],$ an ordered $2$-partition of the values into two parts $\Pi_1$ and $\Pi_2$ with $|\Pi_1| = n$ and $|\Pi_2| = m$ is a partition where all values in $\Pi_1$ are ``ahead'' of all values of $\Pi_2$, while the internal order of $\Pi_1$ and $\Pi_2$ is irrelevant. 
We define $$B(n,m):= \sum\limits_{\text{All }\Pi_1,\Pi_2 \text{ ordered partition of }[n+m]} \theta^{\# \text{crossing inversions of }(\Pi_1, \Pi_2)},$$ where a crossing inversion with respect to $(\Pi_1, \Pi_2)$ is an inversions of the form $(a, b)$ where $a \in \Pi_1$ and $b \in \Pi_2$.

\begin{lemma}
The numbers $B(n,m)$ satisfy 
\begin{equation}\label{mm}
B(n,m) = B(n-1, m) \cdot \theta^m + B(n, m-1),
\end{equation}
and 
\begin{equation}\label{nn}
B(n,m) = B(n-1, m)+ B(n, m-1) \cdot \theta^n,
\end{equation}
with the initial conditions set as $B(0, x) = 1$ and $B(x, 0) = 1$.
\end{lemma}

\begin{proof}
To establish the first recurrence relation, we need to consider two separate cases according to the value $n+m$. 

\textbf{Case 1:} $n+m \in \Pi_1$. Then we delete $n+m$ from $\Pi_1$ and arrive at a partition of $n+m-1$ elements into subsets of size $n-1$ and $m$. The value $n+m$ contributes $\theta^{m}$ to each partition $\Pi_1, \Pi_2$. Thus, it overall contributes $\theta^m \cdot B(n-1, m)$ to the term $B(n,m)$.

\textbf{Case 2:} $n+m \in \Pi_2$. Then we delete $n+m$ from $\Pi_2$ and arrive at a partition of $n+m-1$ elements into subsets of size $n$ and $m-1$. The value $n+m$ does not feature in the multiplier and the contribution to $B(n,m)$ is $B(n, m-1)$.

Similarly, we can consider in which part the element $1$ lies in and obtain the second recurrence relation. The initial conditions are obvious since one part is empty.
\end{proof}

When $\theta = 1$, we have $$B(n, m) = B(n-1, m) + B(n, m-1).$$ 
A straightforward induction argument can be used to prove that 
$$B(n, m) = {n+m \choose n}.$$

It turns out one can also solve the above recurrence relations even when $\theta \neq 1$.

\begin{lemma}
For $\theta \neq 1$, $n, m \ge 1$,
$$B(n, m) = \frac{(1-\theta^{n+m}) \cdot (1-\theta^{n+m-1}) \cdots (1-\theta^{n+1})}{(1-\theta^m) \cdot (1-\theta^{m-1}) \cdots (1-\theta)},$$
and $B(n, 0) = B(0, m) = 1$.
\end{lemma}

\begin{proof}
We subtract from both sides in~\eqref{mm} $\theta^m \cdot$~\eqref{nn} to obtain 
$$(1-\theta^m) \cdot B(n, m) = (1-\theta^{n+m}) \cdot B(n, m-1).$$ 
Since we know that $B(n, 1) = \frac{1-\theta^{n+1}}{1-\theta}$ the claimed result follows.
\end{proof}

Note that 
$$B(n, m) = \frac{P_{n+m} \cdot \ldots \cdot P_{n+2} \cdot P_{n+1}}{P_{m} \cdot \ldots \cdot P_2 \cdot P_1} =: {P_{n+m} \choose P_{n}}. $$

\subsection{Precise Results for Case 1 (and Case 2) from Section~\ref{strategymallows}}\label{a}

In Section~\ref{a}, we use the term {\em $k$-pickable permutation} to describe a permutation which results in a pick using the strategy that rejects the first $k$ candidates and then accepts the next left-to-right second-maximum; we also use the term {\em non-$k$-pickable permutation} to refer to a permutation which is not $k$-pickable.

Define $P_{-1} = P_0 = 0$ and $(0)! = 1$. Let $$T_2(N, k):= \sum\limits_{\text{non-k-pickable permutation }\pi \in S_{N}} \theta^{\# \text{inversions in }\pi}.$$

\begin{lemma}\label{relation-1}
For $N \ge k+1$ and $k \ge 1$, 
$$T_2(N, k) = \theta^{2N-2k} \cdot P_k \cdot P_{k-1} \cdot (P_{N-2})! + \sum\limits_{i = k+1}^{N} \theta^{2N-2i} \cdot B(i-2, N-i) \cdot T_2(i-1, k) \cdot (P_{N-i})!,$$
where $T_2(k, k ) = (P_{k})!$, since no permutation in $S_k$ is $k$-pickable.
\end{lemma}

\begin{proof}
We have to consider two cases depending on the value $N$. For this purpose, let $\pi \in S_N$.

\textbf{Case 1:} $N$ is at a position within $[1, k]$. Then $\pi$ is not $k$-pickable if and only if the value $N-1$ is also at a position in $[1,k]$. The remaining values form an arbitrary permutation. Thus, this case contributes 
$$\theta^{2N-2k} \cdot P_k \cdot P_{k-1} \cdot (P_{N-2})! = (\theta^{N-1} + \ldots + \theta^{N-k}) \cdot (\theta^{N-2} + \ldots + \theta^{N-k}) \cdot (P_{N-2})!$$ 
to $T_2(N, k)$. 

\textbf{Case 2:} $N$ is at a position within $i \in [k+1, N]$. Then $N-1$ must be located before position $i$ and the positions $[1, i-1]$ must form a non-$k$-pickable permutation. Thus, the contribution of this case to $T_2(N, k)$ may be computed as follows. The value $N-1$ gives a factor of $\theta^{N-i}$ for inversions with values in positions $[i+1, N]$, while the remaining values in positions $[1,i-1]$ and $[i+1, N]$ form a partition $\Pi_1, \Pi_2$ of the values $[1, N-2]$ and thus contribute a factor of $B(i-2, N-i)$ to this case. The values in positions $[1, i-1]$ form a non-$k$-pickable permutation and thus contribute $T_2(i-1, k)$. There is no restriction on the values positioned in $[i+1, N]$ and these contribute $(P_{N-i})!$. Moreover, the value $N$ contributes $\theta^{N-i}$. In conclusion, the total contribution from this case (for $k+1 \le i \le N$) equals 
$$\sum\limits_{i = k+1}^{N} \theta^{N-i} \cdot B(i-2, N-i) \cdot T_2(i-1, k) \cdot (P_{N-i})! \cdot \theta^{N-i}.$$
\end{proof}

\begin{remark}\label{t2n0}
When $k = 0$, we have $T_2(2,0) = T_2(1, 0) = 1$ and $$T_2(N,0) = \sum\limits_{i=2}^{N} \theta^{2(N-i)} \cdot \frac{(P_{N-2})!}{(P_{i-2})!} \cdot T_2(i-1, 0).$$
\end{remark}

We can solve the recurrence relation in Lemma~\ref{relation-1} in closed form.

\begin{lemma}\label{t2nk}
We have $T_2(k,k) = (P_{k})!$. For $N \ge k+1$ and $k \ge 1$,
\begin{equation}\label{t2nk-formula}
T_2(N, k) = (P_{k})! \cdot (1+\theta^2 \cdot P_{k-1}) \cdot (1+\theta^2 \cdot P_{k}) \cdot \ldots \cdot (1+\theta^2 \cdot P_{N-2}).
\end{equation}
\end{lemma}

\begin{proof}
We know from Lemma~\ref{relation-1} that $T_2(k,k) = (P_{k})!$. 
We assume the argument is valid for at most $N-1$, and then prove it for $N$.

Again, by Lemma~\ref{relation-1}, we know 
$$T_2(N, k) = \theta^{2N-2k} \cdot P_{k} \cdot P_{k-1} \cdot (P_{N-2})! + \sum\limits_{i = k+1}^{N} \theta^{2N-2i} \cdot B(i-2, N-i) \cdot T_2(i-1, k) \cdot (P_{N-i})!$$

$$= \theta^{2N-2k} \cdot P_{k} \cdot P_{k-1} \cdot (P_{N-2})! + \theta^{2N-2k-2} \cdot \frac{(P_{N-2})!}{(P_{k-1})!} \cdot (P_{k})! + \theta^{2N-2k-4} \cdot \frac{(P_{N-2})!}{(P_{k})!} \cdot (P_{k})! \cdot (1+\theta^2\, P_{k-1})$$
$$+ \ldots + \theta^2 \cdot \frac{(P_{N-2})!}{(P_{N-3})!} \cdot (P_{k})! \cdot (1+\theta^2\, P_{k-1})(1+\theta^2\, P_k) \cdots (1+\theta^2\, P_{N-4}) + (P_{k})! \cdot (1+\theta^2\, P_{k-1})(1+\theta^2\, P_k) \cdots (1+\theta^2\, P_{N-3}).$$

Note that if we add the terms one-by-one, then the first $j$ terms we arrive at are
$$\theta^{2N-2k-2(j-1)} \cdot \frac{(P_{N-2})!}{(P_{k-3+j})!} \cdot (P_{k})! \cdot (1+P_{k-1} \cdot \theta^2) \cdots (1+P_{k-3+j} \cdot \theta^2).$$  
Hence, we obtain 
$$(P_{k})! \cdot (1+P_{k-1} \cdot \theta^2)(1+P_{k} \cdot \theta^2) \cdots (1+P_{N-2} \cdot \theta^2).$$
\end{proof}

\begin{remark}\label{t2n0-formula}
For $k = 0$ and $N \ge 3$, we can solve by Remark~\ref{t2n0} that $$T_2(N, 0) = (1+\theta^2)(1+\theta^2 \cdot P_2) \cdots (1+\theta^2 \cdot P_{N-2}),$$
which agrees with the formula obtained by plugging in $k=0$ to~\eqref{t2nk-formula}.
\end{remark}

Next, we introduce the notion of a {\em $k$-winnable permutation}, corresponding to a permutation such that the global second-best candidate ($N-1$) can be identified using the positional strategy that rejects the first $k$ candidates and accepts the next left-to-right second-maximum thereafter. 

To this end, we define 
$$W_2(N,k) = \sum\limits_{k\text{-winnable }\pi \in S_N} \theta^{\#\text{inversions in }\pi}.$$

\begin{theorem}\label{relation-final}
One has
$$W_2(N, k) = \theta^2 \cdot P_{N-2} \cdot W_2(N-1, k) + \theta \cdot T_2(N-1, k),$$ 
with the initial condition $W_2(k+1, k) = \theta \cdot (P_{k})!.$
\end{theorem}

\begin{proof}
If the last position has the value $N$, then the permutation cannot be $k$-winnable as the value $N-1$ is never going to be picked as a left-to-right second-maximum. Thus, we have two possible scenarios for a $k$-winnable permutation $\pi \in S_N$.

\textbf{Case 1:} The last position contains one of the values $i = 1, 2, \ldots, N-2$. Then it contributes $N-i$ to the inversion count and we may view the remaining values as some $k$-winnable $\tilde{\pi} \in S_{N-1}$. These contribute $W_2(N-1, k) \cdot (\theta^{N-1}+\theta^{N-2}+ \cdots + \theta^2) = \theta^2 \cdot P_{N-2} \cdot W_2(N-1, k)$ to $W_2(N,k)$.

\textbf{Case 2:} The last position is $N-1$. Then the first $N-1$ positions form a non-$k$-pickable permutation. The value $N-1$ at the position $N$ contributes $\theta$.

The initial condition holds because when there are in total $k+1$ positions then the $(k+1)^{\text{th}}$ position must be $k$ and the elements in positions $[1, k]$ can represent any permutation in $S_{k}$.
\end{proof}

\begin{theorem}\label{formula}
For $k \ge 1$, we have 
$$W_2(N, k) = \theta \cdot T_2(N, k) - \theta^{2N-2k+1} \cdot (P_{N-2})! \cdot P_k \cdot P_{k-1}$$ 
$$= \theta \cdot (P_{k})! \cdot \{(1+\theta^2 \cdot P_{N-2})(1+\theta^2 \cdot P_{N-3}) \cdots (1+\theta^2 \cdot P_k) \cdot (1+\theta^2 \cdot P_{k-1}) - \theta^{2N-2k} \cdot P_{N-2} \cdot P_{N-3} \cdots P_{k-1}\}.$$
\end{theorem}

\begin{proof}
By Theorem~\ref{relation-final}, we have the recurrence relation for $$W_2(N,k) \text{ (relation 1)}, W_2(N-1,k)\text{ (relation 2)}, \ldots, W_2(k+2, k) \text{ (relation N-k-1)}.$$ 

Then, we multiply relation 1 with $1$, relation 2 with $\theta^2 \cdot P_{N-2}$, relation j with $\theta^{2j-2} \cdot P_{N-2} \cdot P_{N-3} \cdots P_{N-j}$, $j \in \{3, \ldots, N-k-1\}$. Then we add those equations and use the initial condition $W_2(k+1, k) = \theta \cdot (P_{k})!$ to obtain the desired formula.
\end{proof}

\begin{remark}
For $k=0$, we have $$W_2(N, 0) = \theta \cdot T_2(N, 0).$$ Since the strategy of rejecting no candidate in the beginning and then accepting the next left-to-right second-maximum is the same as rejecting the first candidate and then accepting the next left-to-right second-maximum, the case when $k=0$ is going to be included in the case when $k=1$.
\end{remark}

\begin{theorem}
When $\theta>1$ and $N \to \infty$, the optimal strategy is to reject the first $j=k(\theta)$ candidates, where $k(\theta)$ is a function of $\theta$ that does not depend on $N$, and then select the next left-to-right second-maximum thereafter. 
\end{theorem}
Numerical results for $k(\theta)$ are provided after the proof. 

\begin{proof}
By simplifying the result of Theorem~\ref{formula}, we have 
$$\frac{W_2(N,k)}{(P_{N})!} = \theta \cdot (1- \frac{\theta (\theta - 1)}{\theta^N - 1}) (1- \frac{\theta (\theta-1)}{\theta^{N-1}-1}) \cdots (1- \frac{\theta (\theta-1)}{\theta^{k+1}-1}) - \theta^{2N-2k+1} \cdot \frac{(\theta^k-1)(\theta^{k-1}-1)}{(\theta^N-1)(\theta^{N-1}-1)}.$$

\textbf{Case 1:} $k \to \infty, N \to \infty.$ Then, since $\theta>1$, both the first and second term converge to $\theta$. Thus the limit is $0$.

\textbf{Case 2:} $k \not\to \infty, N \to \infty.$ Then, the second term converges to $\theta \cdot (1-\frac{1}{\theta^k})(1-\frac{1}{\theta^{k-1}})$. The first term converges since $\prod\limits_{j=k+1}^{\infty} (1-\frac{\theta (\theta - 1)}{\theta^j-1})$ converges if and only if $\sum\limits_{j=k+1}^{\infty} \frac{1}{\theta^j-1}$ converges; the latter converges because of the integral test. Thus, the optimal asymptotic probability will occur for some fixed $k(\theta)$.
\end{proof}

Although the infinite product always converges, finding an explicit formula for the probability is hard. Thus, we instead provide some numerical results in Table~\ref{table-2}.

\begin{figure}
\begin{center}
  \includegraphics[scale=0.535]{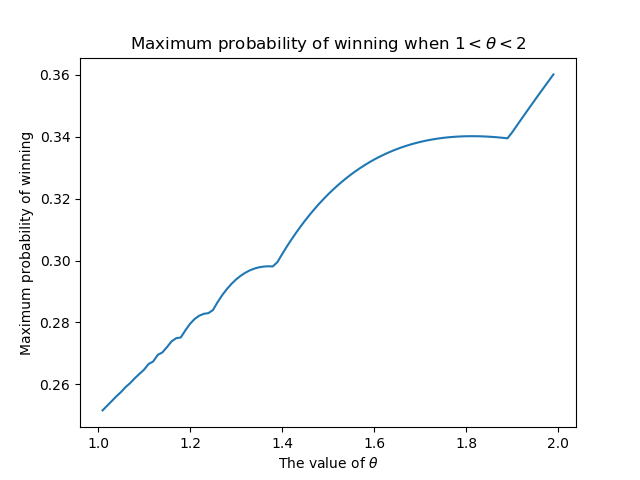}
   \includegraphics[scale=0.535]{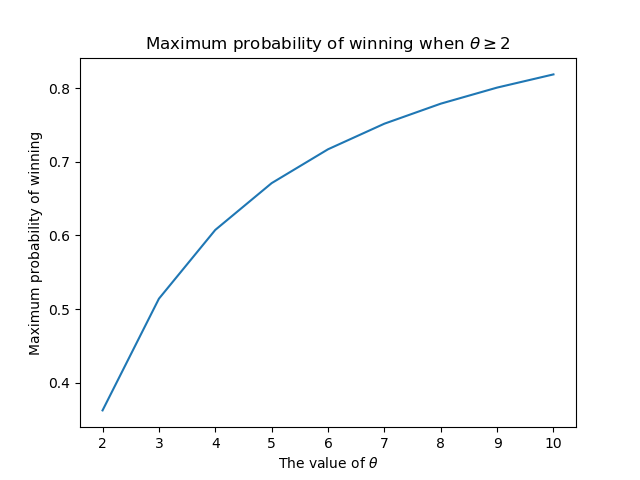}
  \caption{Probability of winning using the optimal strategy for a Mallows distribution with $\theta>1$.}\label{best-1}
\end{center}
\end{figure}
\begin{table}
\begin{center}
\begin{tabular}{|c|c|c|c|c|c|c|c|}
\hline
$\theta$ & reject first $k$  & max probability & $\theta$ & $k$  & max prob  \\ \hline
1.01& 69 &0.25154698   & 1.6 & 2 & 0.33261548    \\ \hline
1.02& 35 &0.25304761   &  1.7&  2 & 0.33832874   \\ \hline
1.03& 24 &0.25456399   &  1.8&  2& 0.34018156   \\ \hline
1.04& 18&0.25609089    &  1.9&  1& 0.34138762  \\ \hline
1.05& 15&0.25746213    &  2&  1& 0.36219565\\ \hline
1.06&12 &0.25906545    &  3&  1& 0.51401101\\ \hline
1.07& 11&0.26037841    &  4&  1& 0.6075226 \\ \hline
1.08& 9&0.26193451     &  5&  1&  0.67111688\\ \hline
1.09& 8&0.26332955     & 6 & 1&  0.71712202\\ \hline
1.10& 8& 0.26468079   & 7&  1& 0.75191395\\ \hline
1.2 &  4 & 0.27951623  &  8 & 1&  0.77912838\\ \hline
1.3 & 3 & 0.29385177   &9&  1& 0.80098779\\ \hline
1.4 & 2& 0.30199267  & 10&  1& 0.81892569  \\ \hline
1.5 & 2& 0.32134993  &     &    &   \\ \hline
\end{tabular}
\end{center}
\caption{Maximum probabilities and optimal strategies for $\theta > 1$.}\label{table-2}
\end{table}

Figure~\ref{best-1} and Table~\ref{table-2} show the optimal success probabilities for various values of $\theta>1$. The maximum winning probability converges to $0.25$ as $\theta \to 1+$, which matches the well known result for the optimal probability $0.25$ when $\theta = 1$.

Note that $k=1$ is optimal for $\theta \ge 1.892$ (approximately), $k = 2$ is optimal for $1.385 \le \theta \le 1.891$ (approximately), and $k=3$ is optimal for $1.247 \le \theta \le 1.384$ (approximately).

The winning probability is increasing and converging to $1$ as $\theta$ increases, when $k=1$ is the optimal; the winning probability is maximized at $(\theta, p) = (1.81, 0.340203)$ when $k=2$ is optimal; the winning probability is maximized at $(\theta, p) = (1.37, 0.298144)$ when $k=3$ is optimal (See Figure~\ref{best-2}). 

Intuitively, we have that the Mallows distribution becomes highly concentrated around the permutation $[N(N-1)\ldots 21]$ when $\theta$ increases, and thus rejecting the first candidate and accepting the next left-to-right second-maximum will capture the value $(N-1)$ most of the times (the probability tends to $1$ as $\theta \to \infty$). However, for $k=2$ and $k=3$, since the distribution concentrates around the permutation $[N(N-1)\ldots 21]$ as $\theta \to \infty$, rejecting the first two or three candidates, respectively, and then accepting the next left-to-right second-maximum is increasingly unlikely to capture the value $N-1$.

\begin{figure}
\begin{center}
  \includegraphics[scale=0.35]{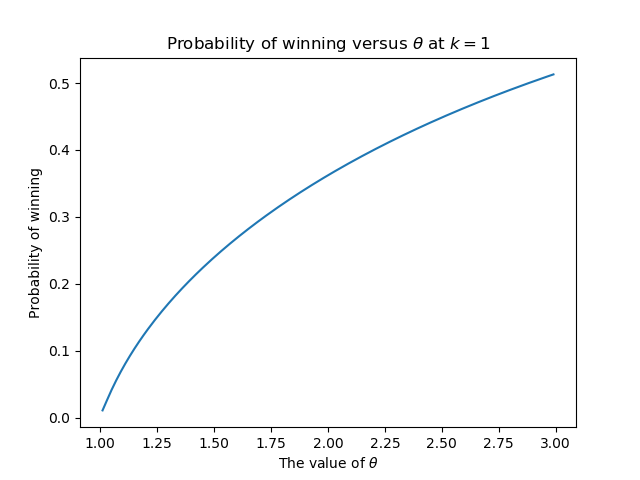}
  \includegraphics[scale=0.35]{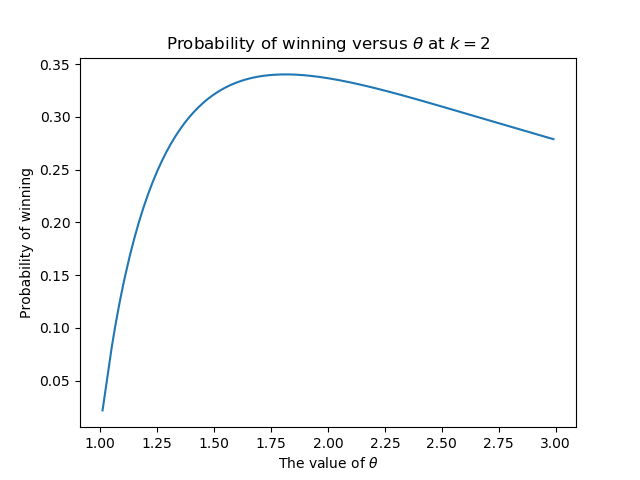}
  \includegraphics[scale=0.35]{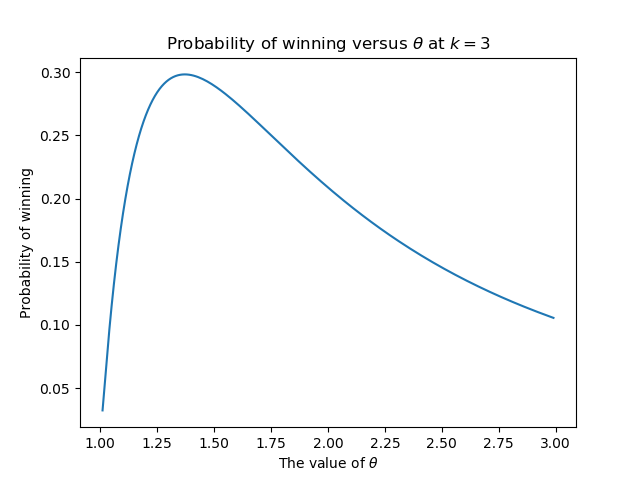}
  \caption{Probability of winning for $\theta > 1$ when we use the strategy of rejecting the first $k$ candidates and then accepting the next second-maximum thereafter, where $k=1,2$ and $3$.}\label{best-2}
\end{center}
\end{figure}

\begin{remark}\label{theta=1solve}
When $\theta = 1$, for $k \ge 1$ we have by Theorem~\ref{formula} 
$$\frac{W_2(N,k)}{N!} = \frac{k}{N} - \frac{k(k-1)}{N(N-1)} = \frac{k (N-k)}{N(N-1)} \text{ and } =1/N \text{ for } k=0.$$ Therefore, the maximum probability of winning is $\frac{N^2}{4N(N-1)} \to 1/4$ and is realized at $k = N/2$. The optimal strategy is to 1) reject the first $\frac{N}{2}-1$ candidates; 2) accept or reject the $\frac{N}{2}$th candidate if it is a left-to-right maximum; reject this candidate otherwise; 3) for a candidate $j > N/2$ we either accept him/her if the candidate is a left-to-right second-maximum; or, we accept or reject the candidate if he/she is a left-to-right maximum; otherwise, we reject the candidate.
\end{remark}

\subsection{Precise result for Case 3.1 (and Case 3.2)}\label{c1}

Unlike in the previous section, in Section~\ref{c1} we use the term {\em $k$-pickable permutation} for a permutation that corresponds to a strategy that rejects the first $k$ candidates and then accepts the next left-to-right maximum and results in one pick. We also use the term {\em non-$k$-pickable permutation} to describe a permutation which is not $k$-pickable. 

Let 
$$T_1(N, k):= \sum\limits_{\text{non-k-pickable permutation }\pi \in S_{N}} \theta^{\# \text{inversions in }\pi}.$$

\begin{lemma}
We have $T_1(N,0) = 0$ and for $k \ge 1$, $$T_1(N,k)=(\theta^{N-1}+ \ldots+\theta^{N-k}) \cdot (P_{N-1})!.$$
\end{lemma}

\begin{proof}
Let $\pi\in S_N$ be non-$k$-pickable. If the value $N$ is positioned in $[k+1, N]$, then we must have one pick. Thus, the value $N$ is positioned in $[1,k]$ and the other positions can be viewed as an arbitrary permutation. If the value $N$ is at position $i\in [1,k]$, it contributes $\theta^{N-i},$ and the remaining terms contribute $(P_{N-1})!$.
\end{proof}

In this subsection, by a {\em $k$-winnable permutation} we mean a permutation such that the global second-best candidate ($N-1$) can be identified using the positional strategy that rejects the first $k$ candidates and accepts the next left-to-right maximum thereafter. We define 
$$W_1^*(N, k) = \sum\limits_{\text{k-winnable }\pi\in S_N} \theta^{\text{\#inversions in }\pi}.$$

\begin{remark}
When $k=0$, the strategy is to accept the first candidate. Thus, we win if and only if the value $N-1$ appears first. Since the value $N-1$ contributes $\theta^{N-2}$ to $W_1^*(N, k)$ and the remaining positions can be viewed as an arbitrary permutation in $S_{N-1}$, the probability of winning is 
$$\frac{W_1^*(N,0)}{(P_{N})!} = \frac{\theta^{N-2}(P_{N-1})!}{(P_{N})!} = \frac{(1-\theta)\cdot \theta^{N-2}}{1-\theta^N} \to 0 \text{ as }N \to \infty.$$ 
\end{remark}

\begin{theorem}\label{solve-2}
For $k \ge 1$,
$$W_1^*(N,k) = \theta^2 \cdot P_{N-2} \cdot W_1^*(N-1,k) + \sum\limits_{i=k+1}^{N-1} \theta^{N-i-1} \cdot T_1(i-1, k) \cdot B(i-1, N-i-1) \cdot (P_{N-i-1})!,$$ 
with initial condition $W_1^*(k+1, k) = 0$ and $W_1^*(k+2, k) = (P_{k})!.$
\end{theorem}

\begin{proof}
If the value in the last position is $N-1$, then the permutation cannot be $k$-winnable as ($N$ appears before $N-1$ and the value $N-1$ is never going to be picked as a left-to-right maximum). Thus, we have to consider two cases for a $k$-winnable permutation $\pi \in S_N$.

\textbf{Case 1:} The last position in $\pi$ is one of the values $i = 1, 2, \ldots, N-2$. This contributes $N-i$ to the inversion count and we may view the remaining entries as some $k$-winnable $\tilde{\pi} \in S_{N-1}$. These contribute $W_1^*(N-1,k) \cdot (\theta^{N-1}+\theta^{N-2}+ \ldots + \theta^2) = \theta^2 \cdot P_{N-2} \cdot W_1^*(N-1,k)$ to $W_1^*(N,k)$.

\textbf{Case 2:} The entry in the last position of $\pi$ is $N$. The value $N-1$ must therefore be in positions $[k+1, N-1]$, say $i$. Then, the entries in positions $[1, i-1]$ form a non-$k$-pickable permutation and there is no restrictions on the values in positions $[i+1, N-1]$. Therefore, when $N-1$ is at position $i \in [k+1, N-1]$, $T_1(i-1,k)$ counts inversions in positions $[1, i-1]$, $B(i-1, N-i-1)$ counts the inversion in between, $(P_{N-i-1})!$ counts the inversions for positions $[i+1, N-1]$, and $\theta^{N-i-1}$ counts the inversions created by the value $N-1$ and values at positions in $[i+1, N-1]$.

When there are $k+1$ values, it is impossible to win using the strategy that rejects the first $k$ positions and accepts the next left-to-right maximum. When there are $k+2$ values, the only case when we can win by rejecting the first $k$ positions and accepting the next left-to-right maximum is when the value in the $(k+1)^{ \text{th}}$ position is $N-1 = k+1$ and the value in the  $k^{ \text{th}}$ position is $N = k+2$, while the remaining positions capture an arbitrary permutation in $S_k$.
\end{proof}

The above recurrence relation can be solved for and the closed form expression is presented in the result below.

\begin{theorem}\label{solvew1}
For $N \ge k+2$,
$$W_1^*(N, k) =  (P_{N-2})! \cdot  \theta^{N-k-2} \cdot (\frac{1-\theta^{N-k-1}}{1-\theta} + \frac{1-\theta^{N-k-2}}{1-\theta} \frac{1-\theta^k}{1-\theta^{k+1}} + \frac{1-\theta^{N-k-3}}{1-\theta} \frac{1-\theta^k}{1-\theta^{k+2}} + \ldots $$
$$+ \frac{1-\theta^2}{1-\theta} \frac{1-\theta^k}{1-\theta^{N-3}} + \frac{1-\theta}{1-\theta} \frac{1-\theta^k}{1-\theta^{N-2}}),$$
and $W_1^*(k+1, k) = 0$ and $W_1^*(k+2, k) = (P_{k})!$.
\end{theorem}

\begin{proof}
We first evaluate the sum in the recurrence relation of Theorem~\ref{solve-2}. We have
$$f(N):=\sum\limits_{i=k+1}^{N-1} \theta^{N-i-1} \cdot T_1(i-1, k) \cdot B(i-1, N-i-1) \cdot (P_{N-i-1})! = \sum\limits_{i=k+1}^{N-1} \theta^{N-i-1} \cdot T_1(i-1,k) \cdot \frac{(P_{N-2})!}{(P_{i-1})!} $$
$$= \theta^{N-k-2} \cdot (P_{k})! \cdot \frac{(P_{N-2})!}{(P_{k})!} + \theta^{N-k-3} \cdot \theta P_k \cdot (P_{k})! \cdot \frac{(P_{N-2})!}{(P_{k+1})!}+ \theta^{N-k-4} \cdot \theta^2\cdot P_k \cdot (P_{k+1})! \cdot \frac{(P_{N-2})!}{(P_{k+2})!}+ \ldots + $$$$\theta^{N-k-2}\cdot P_k \cdot (P_{N-3})! \cdot \frac{(P_{N-2})!}{(P_{N-2})!} = \theta^{N-k-2} \cdot (P_{N-2})! \cdot (1 + \frac{P_k}{P_{k+1}}+ \frac{P_k}{P_{k+2}}  + \ldots + \frac{P_k}{P_{N-2}}).$$

Similarly to what was done in the proof in Theorem~\ref{formula}, we can obtain the stated result after some simplification.
\end{proof}

The game winning probability of our strategy is then $W_1^*(N, k)$ plus the probability that no selection was made before the last position and $N-1$ appears at the last position, which is $$W_1(N,k) = W_1^*(N,k) + \theta \cdot T_1(N-1, k), \text{ where } k \le N-1.$$

\begin{figure}
\begin{center}
  \includegraphics[scale=0.68]{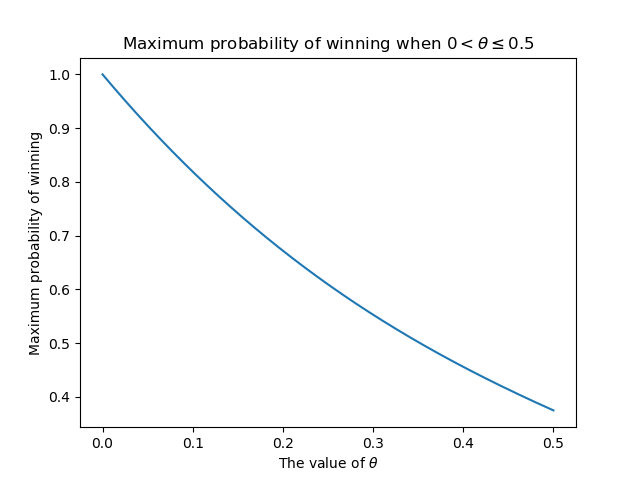}
  \caption{Probability of winning using the optimal strategy for a Mallows distribution with $0<\theta \le \frac{1}{2}$.}\label{best-3}
\end{center}
\end{figure}

\begin{theorem}\label{solution-2}
When $0 < \theta < \frac{1}{2}$, the optimal strategy as $N$ tends to infinity is to reject all but the last two candidates and then accept the next left-to-right maximum and if no selection is made before the last position then accept the last position. The maximum probability of winning is  $(1- \theta)(1-\theta+\theta^2)$ (See Figure~\ref{best-3}).
\end{theorem}

\begin{proof}
The proof is postponed to the Appendix (Section~\ref{appendix}).
\end{proof}

\subsection{Precise result for Case 3.3}\label{c3}

By Case 3.3 described in Section~\ref{strategymallows}, we know that the optimal strategy is a $(k_1, k_2)$-strategy or a $(k_2, k_1)$-strategy such that $0 \le k_1, k_2 \le N-2$. We show in this subsection that $N-k_1 \not \to \infty$ and $N-k_2 \not \to \infty$ for both strategies. By Theorem~\ref{thetamiddle2}, we know $k_1 \le k_2$ and thus only Case 3.3.2 in Section~\ref{strategymallows} can occur.

We call a permutation $\pi \in S_N$ {\em $(k_1, k_2)$-winnable} if it results in a win using the $(k_1, k_2)$-strategy, i.e., by rejecting the first $k_1$ candidates then accept the next left-to-right maximum thereafter or rejecting the first $k_2$ candidates then accept the next left-to-right second-maximum thereafter, whichever appears first. Let $W_1(N, k_1, k_2)$ stand for 
$$\sum\limits_{(k_1,k_2)-\text{winnable permutations }\pi \in S_N}  \theta^{\text{\#inversions in }\pi}.$$

Throughout this section, we call a permutation $\pi \in S_N$ $(k_1, k_2)$-pickable if it results in one selection using the $(k_1, k_2)$-strategy. 

Let $T_1(N,k_1,k_2)$ stand for $$\sum\limits_{\text{non-}(k_1,k_2)-\text{pickable permutations }\pi \in S_N}  \theta^{\text{\#inversions in }\pi}.$$

Recall that we know $T_1(N, 0) = 0$ and for $k_1 \ge 1, N \ge k_1$, $$T_1(N, k_1) = (\theta^{N-1}+ \ldots + \theta^{N-k_1}) \cdot (P_{N-1})! = \theta^{N-k_1} \cdot P_{k_1} \cdot (P_{N-1})!.$$

\begin{lemma}\label{tk1k2}
For $N \ge k_2$,
$$T_1(N, k_1, k_2) = \theta^{2N-k_1-k_2} \cdot P_{k_1} \cdot P_{k_2-1} \cdot (P_{N-2})!.$$
\end{lemma}
\begin{proof}
Let $\pi \in S_N$ be non-$(k_1, k_2)$-pickable. Then the value $N$ must be positioned in $[1, k_1]$ and the value $N-1$ must be positioned in $[1, k_2]$. There are no restrictions on the other values. Thus, we have 

$$T_1(N, k_1, k_2) = (\theta^{N-1}+ \ldots + \theta^{N-k_1}) \cdot (\theta^{N-2}+\theta^{N-3} + \ldots + \theta^{N-k_2}) \cdot (P_{N-2})! = \theta^{2N-k_1-k_2} \cdot P_{k_1} \cdot P_{k_2-1} \cdot (P_{N-2})!.$$
\end{proof}

\begin{theorem}\label{wk1k2}
For $k_1 \le k_2 \le N-2$,
$$W_1(N, k_1, k_2) = \theta^2 \cdot P_{N-2} \cdot W_1(N-1, k_1, k_2) + \theta \cdot T_1(N-1, k_1, k_2) + $$ $$\sum\limits_{i = k_1+1}^{k_2+1} \theta^{N-i-1} \cdot T_1(i-1, k_1) \cdot B(i-1, N-i-1) \cdot (P_{N-i-1})!+ \sum\limits_{i = k_2+2}^{N-1} \theta^{N-i-1} \cdot T_1(i-1, k_1, k_2) \cdot B(i-1, N-i-1) \cdot (P_{N-i-1})!.$$
\end{theorem}

\begin{proof}
The proof is postponed to the Appendix (Section~\ref{appendix}).
\end{proof}

We can solve the recurrence relation in Theorem~\ref{wk1k2} as described in the result to follow.
\begin{theorem}\label{solvewk1k2}
For $k_1 \le k_2 \le N-3$, 
$$W_1(N, k_1, k_2) = \theta^{N-k_1-2} \cdot P_{k_1} \cdot (P_{N-2})! \cdot (\theta^{N-k_2+1} + \theta^{N-k_2+1} \cdot P_{k_2-1} \cdot \sum\limits_{i = k_2}^{N-2} \frac{1}{P_i}$$
$$ + \sum\limits_{i = k_1}^{k_2} \frac{P_{N-i-1}}{P_{i}} + \theta \cdot P_{k_2-1} \cdot \sum\limits_{i = k_2}^{N-3} \frac{\theta^{i-k_2} \cdot P_{N-i-2}}{P_i \cdot P_{i+1}}).$$ 

Moreover,  when $N = k_2+1$ we have
$$ W_1(k_2+1, k_1, k_2) = (P_{k_2-1})! \cdot P_{k_1} \cdot \theta^{k_2-k_1-1} \cdot (\theta^2 + \sum\limits_{k_1}^{k_2-1} \frac{P_{k_2-i}}{P_i}); \text{ and when $N = k_2 + 2$ }$$
$$W_1(k_2+2, k_1, k_2) = \theta^{N-k_1-2} \cdot P_{k_1} \cdot (P_{N-2})! \cdot \left(\theta^{N-k_2+1} + \theta^{N-k_2+1} \cdot \frac{P_{k_2-1}}{P_{k_2}}  + \sum\limits_{i = k_1}^{k_2} \frac{P_{N-i-1}}{P_i}\right).$$
\end{theorem}

\begin{proof}
The proof is postponed to the Appendix (Section~\ref{appendix}).
\end{proof}

We did not consider the case when $k_1 = 0$ since it means that we are using a strategy that accepts the first candidate. The probability of winning with this strategy equals 
$$\frac{\theta^{N-2} \cdot (P_{N-1})!}{(P_{N})!} = \frac{\theta^{N-2} \cdot (1 - \theta)}{1 - \theta^N} \to 0 \text{ as } N \to \infty,\, \frac{1}{2}<\theta<1.$$

\begin{theorem}\label{resultk1k2}
For $\frac{1}{2} < \theta < 1$, $1 \le k_1 \le k_2 \le N-2$, and $N \to \infty$, the optimal $(k_1, k_2)$-strategy is to have $k_1 = k_1(\theta)$ and $k_2 = k_2(\theta)$ for some functions $k_1(\theta)$ and $k_2(\theta)$ such that $N-k_1(\theta) \not \to \infty$ and $N-k_2(\theta) \not \to \infty$ (Numerical results are presented after the proof). 

Define $x := N-k_1 \not \to \infty$ and $y := N-k_2 \not \to \infty$. The probability of winning equals 
$$f(x, y):=\theta^{x-2} \cdot (1 - \theta) \cdot (\theta^{y+1} + \theta^{y+1} \cdot (y-1) + (x - y + 1 - \theta^{y-1} \cdot \frac{1 - \theta^{x-y+1}}{1 - \theta}) + \theta \cdot (\frac{1- \theta^{y-2}}{1-\theta} - \theta^{y-2} \cdot (y - 2))).$$ 
\end{theorem}

\begin{proof}
The proof is postponed to the Appendix (Section~\ref{appendix}).
\end{proof}
We define the {\em $(k_2, k_1)$-strategy} with $k_2 \le k_1 \le N-2$ to be the strategy that rejects the first $k_2$ candidates then accepts the next left-to-right second-maximum thereafter or rejects the first $k_1$ candidates and then accepts the next left-to-right maximum thereafter, whichever appears first. 

We can also similarly define a $(k_2, k_1)$-pickable permutation $\pi \in S_N$, $T_2(N, k_2, k_1)$, and $W_2(N, k_2, k_1)$. By arguments similar to those used in Lemma~\ref{tk1k2}, Theorem~\ref{wk1k2},~\ref{solvewk1k2}, and~\ref{resultk1k2}, we can prove Theorem~\ref{resultk2k1}. The proof is postponed to the Appendix (Section~\ref{appendix}).

\begin{theorem}\label{resultk2k1}
For $\frac{1}{2} < \theta < 1$, $N \to \infty$, the optimal $(k_2, k_1)$-strategy is to have $k_1 = k_1(\theta)$ and $k_2 = k_2(\theta)$ for some functions $k_1(\theta)$ and $k_2(\theta)$ such that $N-k_1(\theta) \not \to \infty$ and $N-k_2(\theta) \not \to \infty$.
\end{theorem}

By Theorem~\ref{thetamiddle2}, every Type I prefix of length longer than $k_2(\theta)$ is positive. Therefore, we have $k_1(\theta) \le k_2(\theta)$ and conclude that the optimal strategy is the $(k_1(\theta), k_2(\theta))$-strategy, i.e., we reject the first $k_1(\theta)$ candidates and then accept the next left-to-right maximum thereafter or reject the first $k_2(\theta) \ge k_1(\theta)$ candidates and then accept the next left-to-right second-maximum thereafter, whichever appears first.

Since for $f(x, y)$ as defined in Theorem~\ref{resultk1k2} we have that $x$ and $y$ must both be integers, and since $f(x, y) \to 0$ as $x \to \infty$ and $y \to \infty$, we can pick a large number (say, $100$) as an upper bound for $x$ and $y$; and, for each $\theta \in \{0.51, 0.52, \ldots, 0.99\}$ use brute force search to find the maximum of $f(x,y)$ subject to the constraint $1 \le y \le x \le 100$. (The number $100$ is large enough as we also ran computer simulations to find the maximum of $f(x,y)$ subject to $1 \le y \le x$ without restricting ourselves to integer values of 
$x$ and $y$; it turns out that the $x,y$ which realize the maximum of $f(x,y)$ obtained with integer constraints are floors or ceilings of the $x,y$ that maximize $f(x,y)$ without the integer constraints). 

The optimal strategy is a $(k_1, k_2)$-strategy for some $k_1 \le k_2$ such that both $N-k_1$ and $N-k_2$ $\not\to \infty$ (See Figure~\ref{best-4} and Table~\ref{table-1}). Note that as $x \to \infty$, $y \to \infty$, we have that the probability of winning $\to 0.25$ as $\theta \to 1$, which matches the well-known result for $\theta = 1$ and also the (same and more detailed) result by our approach presented in Remark~\ref{theta=1solve}.

\begin{figure}
\begin{center}
  \includegraphics[scale=0.7]{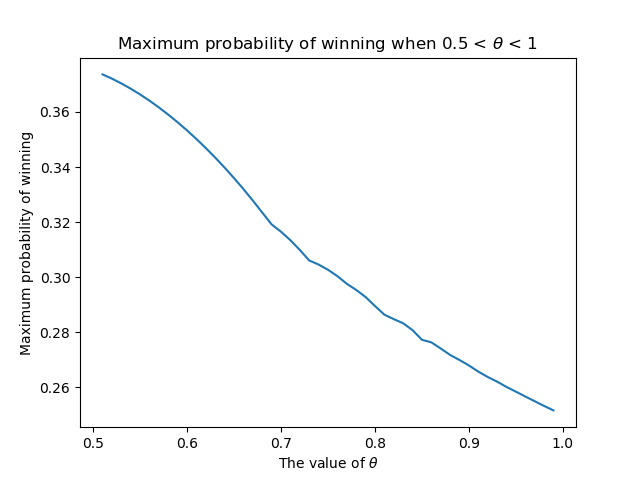}
  \caption{Probability of winning using the optimal strategy for a Mallows distribution with $\frac{1}{2}<\theta<1$.}\label{best-4}
\end{center}
\end{figure}

\begin{table}
\begin{center}
\begin{tabular}{|c|c|c|c|c|c|c|c|}
\hline
$\theta$ &  $x = N-k_1$ & $y = N-k_2$ & $f(x,y)$ & $\theta$ &  $x = N-k_1$ &  $y = N-k_2$ & $f(x,y)$    \\ \hline
0.51 & 3 & 2 & 0.37365098 & 0.76 & 4 & 3 & 0.30035513 \\ \hline
0.52 & 3 & 2 & 0.37210767 & 0.77 & 4 & 3 & 0.29758801 \\ \hline
0.53 & 3 & 2 &  0.37037533& 0.78 & 5 & 3 & 0.29534636 \\ \hline
0.54 & 3 & 2 & 0.36845868 & 0.79 & 5 & 3 & 0.29278142 \\ \hline
0.55 & 3 & 2 & 0.36636187 & 0.80 & 5 & 3 & 0.28950528 \\ \hline
0.56 & 3 & 2 & 0.36408852 & 0.81 & 5 & 4 & 0.28636405 \\ \hline
0.57 & 3 & 2 & 0.36164162 & 0.82 & 5 & 4 & 0.28475072 \\ \hline
0.58 & 3 & 2 & 0.35902353 & 0.83 & 6 & 4 & 0.28323769 \\ \hline
0.59 & 3 & 2 & 0.35623597 & 0.84 & 6 & 4 &  0.2807399 \\ \hline
0.60 & 3 & 2 & 0.35328 & 0.85 & 6 & 5 & 0.27723561 \\ \hline
0.61 & 3 & 2 & 0.35015597 & 0.86 & 7 & 5 & 0.27631243 \\ \hline
0.62 & 3 & 2 & 0.34686351 & 0.87 & 7 & 5 & 0.27407495 \\ \hline
0.63 & 3 & 2 & 0.34340152 & 0.88 & 8 & 6 & 0.27172552 \\ \hline
0.64 & 3 & 2 & 0.33976812 & 0.89 & 8 & 6 &  0.26989821 \\ \hline
0.65 & 3 & 2 &  0.33596062 & 0.90 & 9 & 7 & 0.26791563 \\ \hline
0.66 & 3 & 2 & 0.33197556  & 0.91 & 10 & 8 & 0.26567038 \\ \hline
0.67 & 3 & 2 & 0.32780861 & 0.92 & 11 & 9 & 0.26372892 \\ \hline
0.68 & 3 & 2 & 0.32345457 & 0.93 & 12 & 10 & 0.26203596 \\ \hline
0.69 & 4 & 2 & 0.31915211 & 0.94 & 14 & 11 & 0.2601134 \\ \hline
0.70 & 4 & 2 & 0.316491 & 0.95 & 17 & 14 & 0.25839363 \\ \hline
0.71 & 4 & 2 & 0.31340159 & 0.96 & 21 & 17 & 0.25663997 \\ \hline
0.72 & 4 & 2 & 0.30987746 & 0.97 & 27 & 23 & 0.25492095 \\ \hline
0.73 & 4 & 3 & 0.30605788 & 0.98 & 39 & 35 & 0.25320664 \\ \hline
0.74 & 4 & 3 & 0.30456693 & 0.99 & 76 & 69 & 0.25158519 \\ \hline
0.75 & 4 & 3 & 0.30267334 &         &      &      &    \\ \hline
\end{tabular}
\end{center}
\caption{Maximum winning probabilities and optimal strategies for $\frac{1}{2}< \theta < 1$.}\label{table-1}
\end{table}

The results of our analysis are summarized in~Figure~\ref{best-5}, depicting the maximum probability of winning versus the value of $\theta>0$.
\begin{figure}
\begin{center}
  \includegraphics[scale=0.5]{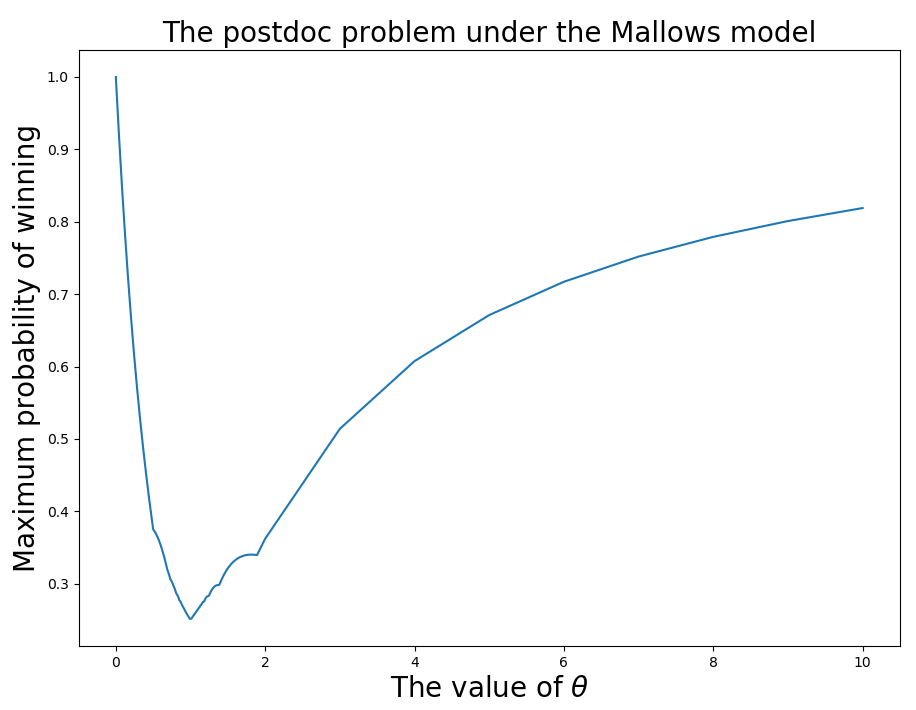}
  \caption{Probability of winning using the optimal strategy for a Mallows distribution with $\theta>0$.}\label{best-5}
\end{center}
\end{figure}

\begin{remark}
An interesting open question is to find the optimal strategy for identifying the $k^{th}$-best candidate when candidates are presented according to the Mallows distribution. We believe that the problem may be addressed using similar proof techniques.
\end{remark}

\textbf{Acknowledgment.} The work was supported in part by the NSF grants NSF CCF 15-26875 and The Center for Science of Information at Purdue University, under contract number 239 SBC PURDUE 4101-38050. Early parts of the work were also  supported by the DARPA Molecular Informatics Program. 


\section{Appendix}\label{appendix}

\textbf{Proof of Theorem~\ref{thetamiddle2}:} Since $N-k \not \to \infty$, we know that $k \to \infty$ and ~\eqref{eq1},~\eqref{eq2},~\eqref{eq3},~\eqref{eq4} become

\begin{equation}\label{eq1'}
Q^o_1(k-1) = \bar{Q}_1(k) + \bar{Q}_2(k) + Q_1^o(k) \cdot \frac{\theta^2}{1-\theta},
\end{equation}
\begin{equation}\label{eq2'}
Q_1(k-1) = Q_1(k) \cdot \frac{\theta}{1-\theta} + \frac{1}{\theta} \cdot Q_2(k),
\end{equation}
\begin{equation}\label{eq3'}
Q_2^o(k-1) = \bar{Q}_1(k) \cdot \theta + \bar{Q}_2(k) \cdot \theta + Q_2^o(k) \cdot \frac{\theta^2}{1-\theta},
\end{equation}
\begin{equation}\label{eq4'}
Q_2(k-1) = Q_2(k) \cdot \frac{\theta^2}{1-\theta},
\end{equation}
respectively. When $N-k = 0$, $Q_2(N) = \theta > 0 = Q_2^o(N)$ and $Q_1(N) = 0 \ge 0 = Q_1^o(N)$.

We prove the Theorem by induction. We first examine the base case $N-k = 1$, for which
$$Q_2(N-1) = \theta^N + \ldots + \theta^3 > \theta^2 = Q_2^o(N-1)$$ 
since $\frac{1}{2} < \theta < 1$ and $Q_1(N-1) = \theta^0 = 1 > \theta = Q_1^o(N-1)$.

Now we assume the argument applies for $\hat{k} \ge N-(j-1)$ and prove it for $k' = N-j$. By the induction hypothesis, we have $Q_2(k') = Q_2(N-j) > (\ge) \text{ } Q_2^o(N-j) = Q_2^o(k')$.
\begin{claim}
Under the assumptions 1) $j = N-k' \not \to \infty$, 2) for all $N-1 \ge \hat{k} \ge N-(j-1)$, $Q_2(\hat{k}) > (\ge) \text{ } Q_2^o(\hat{k})$ and $Q_1(\hat{k}) > (\ge) \text{ } Q_1^o(\hat{k})$, and 3) $Q_2(N-j) > (\ge) \text{ } Q_2^o(N-j)$, we have 
$$Q_1^o(N-j) = \frac{j \cdot \theta^{2j-1}}{(1-\theta)^{j-1}} + \frac{\theta^{2j-4} \cdot P_{1} + \ldots + \theta^{j-2} \cdot P_{j-1}}{(1-\theta)^{j-2}} \text{ and }  Q_2^o(N-j) = \theta \cdot Q_1^o(N-j) .$$
\end{claim}
\begin{proof}
By the assumptions of the claim,~\eqref{eq1'} and~\eqref{eq3'} reduce to
\begin{equation}\label{eq1''}
Q_1^o(k-1) = Q_1(k) + Q_2(k) + Q_1^o(k) \cdot \frac{\theta^2}{1-\theta},
\end{equation}
\begin{equation}\label{eq3''}
Q_2^o(k-1) = Q_1(k) \cdot \theta + Q_2(k) \cdot \theta + Q_2^o(k) \cdot \frac{\theta^2}{1-\theta},
\end{equation}
and we can solve the recurrence using the formulas for $Q_1$ and $Q_2$ provided in Claim~\ref{q2} and Claim~\ref{q1} with the additional conditions that $N \to \infty$ and $N-k \not \to \infty$, and some simple algebra (which we omitted). 

\end{proof}

Under the assumption $Q_2(N-j) > (\ge) \text{ } Q_2^o(N-j)$, we define $h_1(j)$ as
\begin{equation}\label{root2}
h_1(j) := (j+1) \cdot \theta^{j+2} - j \cdot \theta^{j+1} - (j-1) \cdot \theta^j + j \cdot \theta^{j-1} - 1
\end{equation}
so that $h_1(j)  > (\ge) \text{ } 0$. We want to show that $Q_1(N-j) > (\ge) \text{ } Q_1^o(N-j)$, which is equivalent to $h_2(j) > (\ge) \text{ } 0$ with
\begin{equation}\label{root1}
h_2(j) := j \cdot \theta^{j+2} - (j+1) \cdot \theta^{j+1} - (j-1) \cdot \theta^j + j \cdot \theta^{j-1} + \theta - 1.
\end{equation}

To complete the proof, we need to establish Claim~\ref{calc}.

\begin{claim}\label{calc}
For any given $\frac{1}{2} < \theta < 1$, both equalities $h_1(j) = 0$ and $h_2(j) = 0$ (with $h_1(j), h_2(j)$ defined in~\eqref{root2} and~\eqref{root1}) have exactly one real positive root, which we write as $j_1(\theta)$ and $j_2(\theta)$; the inequalities $h_1(j)  > (\ge) \text{ } 0$ and $h_2(j)  > (\ge) \text{ } 0$ hold for $0<j< (\le) \text{ }j_1(\theta)$ and $0<j < (\le) \text{ } j_2(\theta)$, respectively. Moreover, we always have $j_2(\theta)>j_1(\theta)$.
\end{claim}
\begin{proof}

We first show that both equations have exactly one root. The root of~\eqref{root2} may be viewed as the intersection of $f_1(j) = j \cdot (\theta^3-\theta^2-\theta+1) + \theta^3+\theta$ and $f_2(j) = (\frac{1}{\theta})^{j-1}$. The function $f_1$ is linear in $j$, with positive slope $\theta^3-\theta^2-\theta+1 = (\theta-1)^2 \cdot (\theta+1)$ and positive intersection value $\theta^3+\theta$ (at $j=0$); The function $f_2$ is an exponential function with base $1/\theta > 1$ and intersection value $\theta < \theta^3 + \theta$ (at $j=0$). Therefore, $f_1$ and $f_2$ only have one real positive intersection and $h_1>0$ for $0 < j < j_1(\theta)$ (by graphing). 

Similarly, the root of~\eqref{root1} may be viewed as the intersection of $g_1(j) = j \cdot (\theta^3-\theta^2 - \theta + 1) + \theta - \theta^2$ and $g_2(j) = (\theta - \theta^2) \cdot (\frac{1}{\theta})^j$. We have that their intersection at $j=0$ are both equal to $\theta - \theta^2$ and the slope for $g_1$ at $j=0$ is $\theta^3 - \theta^2 - \theta + 1 = (\theta - 1)^2 \cdot (\theta + 1) $, which is always larger than the slope of $g_2$ at $j = 0$, i.e., $ - \ln{\theta} \cdot (\theta - \theta^2)$ for $\frac{1}{2}<\theta<1$. The 
proof is omitted, as it follows from simple calculus. Therefore, $g_1$ and $g_2$ only have one real positive intersection and $h_2>0$ for $0 < j < j_2(\theta)$ (by graphing). 

We next show that $h_2(j_1(\theta)) > 0$ for every $\frac{1}{2}< \theta < 1$. Setting $h_1 = 0$, and plugging $j_1(\theta)$ into $h_2$ shows that $h_2(j_1(\theta)) > 0$ is equivalent to
\begin{equation}\label{root1'}
1 - \theta^{j_1(\theta)} - \theta^{j_1(\theta)+1} > 0, \text{ for all }1/2<\theta<1.
\end{equation}
Let $h_3(j) = 1 - \theta^{j} - \theta^{j+1}$. The function $h_3$ has exactly one positive real root $j_3(\theta) = - \frac{\ln{(1+\theta)}}{\ln{\theta}}$ and $h_3(0) = - \theta < 0$ when $\frac{1}{2}<\theta<1$. If we can show that $j_3(\theta) > j_1(\theta)$ for all $1/2 < \theta < 1$ then~\eqref{root1'} is true; equivalently, if we can show that $h_1(j_3(\theta)) > 0$, then our claim follows. 

To prove this, we show that 
$$h_1(j_3(\theta)) = \frac{(j_3(\theta)+1) \cdot \theta^3 - j_3(\theta) \cdot \theta^2 - (j_3(\theta)-1) \cdot \theta + j_3(\theta) - \theta - \theta^2}{\theta(1+\theta)} $$
$$= \frac{(\theta - 1) \cdot (j_3(\theta) \cdot \theta^2 + \theta^2 - j_3(\theta))}{\theta(1+\theta)} > 0,$$ i.e. $j_3(\theta) \cdot \theta^2 + \theta^2 - j_3(\theta) < 0$, 
which is equivalent to showing that $j_3(\theta) =   - \frac{\ln{(1+\theta)}}{\ln{\theta}} > \frac{\theta^2}{1-\theta^2}$. We omit the proof of the fact that $$- \frac{\ln{(1+\theta)}}{\ln{\theta}} > \frac{\theta^2}{1-\theta^2}, \text{ when } 1/2< \theta < 1;$$ simple calculus or computer verification can be used to verify 
that this fact is true.
\end{proof}\qed


\textbf{Proof of Theorem~\ref{solution-2}:} From Theorem~\ref{solve-2} and by simplifying the previous expressions, we arrive at
$$\frac{W_1(N,k)}{(P_{N})!} = \theta^{N-k-2} \cdot \frac{(1-\theta)^2}{(1-\theta^N)(1-\theta^{N-1})} \cdot (\frac{1-\theta^{N-k-1}}{1-\theta} + \frac{1-\theta^{N-k-2}}{1-\theta} \frac{1-\theta^k}{1-\theta^{k+1}} + \frac{1-\theta^{N-k-3}}{1-\theta} \frac{1-\theta^k}{1-\theta^{k+2}} + \ldots $$
$$+ \frac{1-\theta^2}{1-\theta} \frac{1-\theta^k}{1-\theta^{N-3}} + \frac{1-\theta}{1-\theta} \frac{1-\theta^k}{1-\theta^{N-2}}) + \theta^{N-k} \cdot \frac{(1 - \theta)(1 - \theta^k)}{(1-\theta^N)(1-\theta^{N-1})}.$$

Let $N \to \infty$ and consider the following two cases.

\textbf{Case 1:} $N-k \to \infty$. Since $\theta< \frac{1}{2}$, $\theta^{N-k-2} \to 0$ exponentially, $\frac{(1-\theta)^2}{(1-\theta^N)(1-\theta^{N-1})} \to (1- \theta)^2 \le 1$, and each term in the sum within the parentheses is bounded above by $2$; thus, the sum under parentheses is bounded by $2(N-k-1)$, and $(1 - \theta)(1 - \theta^k) \le 2$. Hence, $\frac{W_1(N, k)}{(P_{N})!} \to 0$.

\textbf{Case 2:} $N-k \not\to \infty$. Then we must have $k \to \infty$. If $N-k=1$ then $\frac{W_1^*(N, k)}{(P_{N})!}$ is zero since the strategy ``reject all but the last candidate and then pick the next left-to-right maximum'' only makes a selection when $N$ appears in the last position, and this selection results in a loss; thus, if $N-k=1$ the best strategy is to simply just accept the last candidate and in this case the probability of winning is $$\theta \cdot \frac{(P_{N-1})!}{(P_{N})!} = \theta \cdot \frac{1-\theta}{1-\theta^{N}} \to \theta (1 - \theta) \text{ as }N \to \infty.$$

For $N-k \ge 2$,
$$\frac{W_1^*(N,k)}{(P_{N})!} \to \theta^{N-k-2} \cdot (1-\theta)^2 \cdot (\frac{1-\theta^{N-k-1}}{1-\theta} + \frac{1-\theta^{N-k-2}}{1-\theta}  + \ldots + \frac{1-\theta^2}{1-\theta}  + \frac{1-\theta}{1-\theta} ) + \theta^{N-k} \cdot (1 - \theta)$$
$$=  \theta^{N-k-2} \cdot (1-\theta) \cdot (N-k-1 - \theta \cdot \frac{1-\theta^{N-k-1}}{1-\theta}) + \theta^{N-k} \cdot (1 - \theta) \overset{\mathrm{x:=N-k}}{=} \theta^{x-2} \cdot ((x-1)(1-\theta) - \theta \cdot (1-\theta^{x-1}))+\theta^x \cdot (1-\theta).$$

Since when $N-k = 2$, the probability converges to $(1- \theta)^2 + \theta^2(1-\theta)> \theta(1- \theta)$ when $0 < \theta < \frac{1}{2}$. Thus, we only need to consider the cases when $N-k \ge 2$. Let $$f(x, \theta) := \theta^{x-2} \cdot ((x-1)(1-\theta) - \theta \cdot (1-\theta^{x-1})) + \theta^x \cdot (1-\theta).$$

\begin{claim}\label{optimize}
For $x \ge 3$ and fixed $0<\theta< \frac{1}{2}$, we have $$f(x+1, \theta) - f(x, \theta) < 0.$$
\end{claim}
\begin{proof}
Since $0<\theta<\frac{1}{2}$ and $x \ge 3$, $$f(x+1, \theta) - f(x, \theta) = x(1-\theta)\theta^{x-1} - \theta^x + \theta^{2x} - (x-1)(1-\theta)\theta^{x-2}+\theta^{x-1}-\theta^{2x-2} + (\theta^{x+1} - \theta^x)(1-\theta)$$
$$= \theta^{x-2} \cdot (1-\theta) \cdot (-x(1-\theta)+(1+\theta) - (1+\theta)\cdot \theta^x) + (\theta^{x+1} - \theta^x)(1-\theta) < \theta^{x-2} \cdot (1-\theta) \cdot ((1+\theta) - x(1-\theta)) + (\theta^{x+1} - \theta^x)(1-\theta)$$
$$\le \theta^{x-2} \cdot (1-\theta) \cdot ((1+\theta) - 3(1-\theta))+ (\theta^{x+1} - \theta^x)(1-\theta) = \theta^{x-2} \cdot (1-\theta) \cdot (4 \theta - 2)+ (\theta^{x+1} - \theta^x)(1-\theta) < 0.$$
\end{proof}
Thus we only need to compare $f(2, \theta)$ and $f(3, \theta)$. It turns out that 
$$f(2, \theta) - f(3, \theta) = (1-\theta+\theta^2)(1-\theta)-(2\theta - \theta^2)(1-\theta) =(1-2\theta)(1-\theta)^2 >0,$$ for all $0 < \theta < \frac{1}{2}$.  Therefore, the optimal strategy is to reject all but the last two candidates and then accept the next left-to-right maximum; if no selection is made before the last position then the only option is to accept the last position. The maximum probability of winning is $(1- \theta)(1-\theta+\theta^2)$ (See Figure~\ref{best-3}). 

\hfill \qed

\vspace{5mm}

\textbf{Proof of Theorem~\ref{wk1k2}:} We need to consider three cases depending on the value in the last position.

\textbf{Case 1:} The last position has a value $i \in \{1, 2, \ldots, N-2\}$. These contribute $\theta^{N-i}$ to the sum. The positions in $[1, N-1]$ may be viewed as a $(k_1, k_2)$-winnable permutation in $S_{N-1}$. Thus, together, these contribute $(\theta^{N-1} + \theta^{N-2} + \ldots + \theta^2) \cdot W_1(N-1, k_1, k_2)$ to the expression of interest.

\textbf{Case 2:} The last position is occupied by the value $N-1$. This contributes $\theta$ to the sum. The elements at positions $[1, N-1]$ must form a non-$(k_1, k_2)$-pickable permutation, with a contribution of $T_1(N-1, k_1, k_2)$.

\textbf{Case 3:} The last position is occupied by the value $N$. Then the value $N-1$ is at a position $i \in [k_1+1, N-1]$. The positions $[1, i-1]$ form a non-$(k_1,k_2)$-pickable permutation if $i \ge k_2+2$, which contributes $T_1(i-1, k_1, k_2)$ to the expression of interest; and, a non-$k_1$-pickable permutation if $k_1+1 \le i \le k_2+1$, which contributes $T_1(i-1, k_1)$ in this case. Furthermore, $B(i-1, N-i-1)$ counts the inversions in-between. There are no restrictions on the positions in $[i+1, N-1]$ and these contribute $(P_{N-i-1)!}$. The value $N-1$ at position $i$ contributes $\theta^{N-i-1}$. \qed

\vspace{5mm}

\textbf{Proof of Theorem~\ref{solvewk1k2}:} For $N = k_2+1$, the only differences between the sets of $(k_1, k_2)$-winnable and $k_1$-winnable permutations are those permutations with value $k_2$ (which equals $N-1$ in this case) at position $k_2+1$ (which equals $N$ in this case) and the value $k_2+1$ (which equals $N$ in this case) at a position in $[1, k_1]$. Thus, we have 
$$W_1(k_2+1, k_1, k_2) = W_1(k_2+1, k_1) + (\theta^{k_2} + \ldots + \theta^{k_2+1-k_1}) \cdot (P_{k_2-1})!.$$ 

Similarly to what we did for the proof of Theorem~\ref{solvew1} we can solve for the case when $N \ge k_2+2$. This completes the proof.\qed

\vspace{5mm}

\textbf{Proof of Theorem~\ref{resultk1k2}:} Let $N \to \infty$. Then 
$$\frac{W_1(N, k_1, k_2)}{(P_{N})!} = \frac{\theta^{N-k_1-2} \cdot P_{k_1}}{P_{N-1}\cdot P_{N}} \cdot (\theta^{N-k_2+1} + \theta^{N-k_2+1} \cdot P_{k_2-1} \cdot \sum\limits_{i = k_2}^{N-2} \frac{1}{P_i}$$
$$ + \sum\limits_{i = k_1}^{k_2} \frac{P_{N-i-1}}{P_i} + \theta \cdot P_{k_2-1} \cdot \sum\limits_{i = k_2}^{N-3} \frac{\theta^{i-k_2} \cdot P_{N-i-2}}{P_i \cdot P_{i+1}}) \to \theta^{N - k_1 - 2} \cdot (1 - \theta^{k_1}) \cdot (1 - \theta) \cdot$$ 
$$ \left(\theta^{N-k_2+1} + \theta^{N-k_2+1} \cdot (\frac{1-\theta^{k_2-1}}{1-\theta^{k_2}} + \ldots + \frac{1-\theta^{k_2-1}}{1-\theta^{N-2}}) + (\frac{1-\theta^{N-k_1-1}}{1-\theta^{k_1}} + \ldots + \frac{1-\theta^{N-k_2-1}}{1-\theta^{k_2}}\right)$$

\begin{equation}\label{winprobk1k2-1}
+ \theta \cdot (1-\theta^{k_2-1}) \cdot \left(\frac{1 \cdot (1 - \theta^{N-k_2-2})}{(1 - \theta^{k_2})\cdot (1 - \theta^{k_2+1})} + \frac{\theta \cdot (1 - \theta^{N-k_2-3})}{(1 - \theta^{k_2+1}) \cdot (1 - \theta^{k_2+2})} + \ldots + \frac{\theta^{N-k_2-3} \cdot (1-\theta)}{(1 - \theta^{N-3}) \cdot (1 - \theta^{N-2})})\right). 
\end{equation}

\textbf{Case 1:} $N - k_1 \to \infty$. Since $\frac{1}{2}< \theta < 1$, $\theta^{N-k_1-2} \to 0$ exponentially. We hence have~\eqref{winprobk1k2-1} is upper-bounded by
$$\theta^{N-k_1-2} \cdot (1 + 1 \cdot (N - k_2 - 1)) + \frac{k_2-k_1+1}{1 - \theta} + \frac{N-k_2-2}{(1-\theta)^2}) \to 0.$$

\textbf{Case 2:} $N - k_1 \not\to \infty$. Then $N - k_2 \not\to \infty$, $k_1 \to \infty$, and $k_2 \to \infty$. 

Let $x = N - k_1$ and $y = N - k_2$. Then~\eqref{winprobk1k2-1} converges to
$$f(x,y) = \theta^{x-2} \cdot (1 - \theta) \cdot (\theta^{y+1} + \theta^{y+1} \cdot (y-1) + (x - y + 1 - \theta^{y-1} \cdot \frac{1 - \theta^{x-y+1}}{1 - \theta}) + \theta \cdot (\frac{1- \theta^{y-2}}{1-\theta} - \theta^{y-2} \cdot (y - 2))).$$ \qed

\textbf{The Proof of Theorem~\ref{resultk2k1}:}

We call a permutation $\pi \in S_N$ {\em $(k_2, k_1)$-winnable} if it results in a win using the $(k_2, k_1)$-strategy. Let $W_2(N, k_2, k_1)$ denote $$\sum\limits_{(k_2,k_1)-\text{winnable permutations }\pi \in S_N}  \theta^{\text{\#inversions in }\pi}.$$

Throughout this subsection, we call a permutation $\pi \in S_N$ $(k_2, k_1)$-pickable if it results in at least one selection using the $(k_2, k_1)$-strategy. Let $T_2(N,k_2,k_1)$ denote $$\sum\limits_{\text{non}-(k_2,k_1)-\text{pickable permutations }\pi \in S_N}  \theta^{\text{\#inversions in }\pi}.$$

Recall that by Lemma~\ref{t2nk} and Remark~\ref{t2n0-formula} we have $T_2(k,k) = (P_k)!$ and when $N \ge k+1$ and $k \ge 0$, 
\begin{equation}\label{t2-appendix}
T_2(N, k) = (P_k)! \cdot (1+\theta^2 \cdot P_{k-1}) \cdot (1+\theta^2 \cdot P_k) \cdot \ldots \cdot (1+\theta^2 \cdot P_{N-2}).
\end{equation}

\begin{lemma}
For $N \ge k_1$,
$$T_2(N, k_2, k_1) = \theta^{2N-2k_1} \cdot \frac{(P_{N-2})!}{(P_{k_1-2})!} \cdot T_2(k_1, k_2).$$
\end{lemma}
\begin{proof}
Let $\pi \in S_N$ be non-$(k_2, k_1)$-pickable. Then the value $N$ must be in $[1, k_1]$ since otherwise it would be picked as a left-to-right maximum if no selection was made before. The value $N-1$ must also be in $[1, k_1]$ since otherwise it would be picked as a left-to-right second-maximum if no selection was made before. 

The positions $[1, k_1]$ form a non-$k_2$-pickable permutation and thus positions $[1, k_1]$ contribute $T_2(k_1, k_2) = (P_k)! \cdot (1+\theta^2 \cdot P_{k-1}) \cdot (1+\theta^2 \cdot P_k) \cdot \ldots \cdot (1+\theta^2 \cdot P_{k_1-2})$. Furthermore, each of the values $N$ and $N-1$ contribute $\theta^{N-k_1}$ for the inversions involving positions $[k_1+1, N]$. There are no restrictions for positions in $[k_1+1, N]$, and thus they contribute $(P_{N-k_1})!$. Moreover, $B(k_1-2, N-k_1)$ counts the in-between inversions. Thus, together we have $$T_2(N, k_2, k_1) = \theta^{2N-2k_1} \cdot T_2(k_1, k_2) \cdot B(k_1-2, N-k_1) \cdot (P_{N-k_1})! = \theta^{2N-2k_1} \cdot \frac{(P_{N-2})!}{(P_{k_1-2})!} \cdot T_2(k_1, k_2).$$
\end{proof}

\begin{lemma}\label{w2k1k2}
For $k_2 \le k_1 \le N-2$,
$$W_2(N, k_2, k_1) = \theta^2 \cdot P_{N-2} \cdot W_2(N-1, k_2, k_1) + \theta \cdot T_2(N-1, k_2, k_1) +$$
$$ \sum\limits_{i = k_1+1}^{N-1} \theta^{N-i-1} \cdot T_2(i-1, k_2, k_1) \cdot B(i-1, N-i-1) \cdot [N-i-1]!.$$
\end{lemma}

\begin{proof}
We have to address three cases depending on the value of the last position.

\textbf{Case 1:} The last position has values $i \in \{1, 2, \ldots, N-2\}$. They contribute $\theta^{N-i}$ and the positions in $[1, N-1]$ can be viewed as a $(k_2, k_1)$-winnable permutation in $S_{N-1}$. Thus, together, these contribute $(\theta^{N-1} + \theta^{N-2} + \ldots + \theta^2) \cdot W_2(N-1, k_2, k_1).$

\textbf{Case 2:} The last position is the value $N-1$. The value $N-1$ at the last position contributes $\theta$. Then the positions $[1, N-1]$ must form a non-$(k_2, k_1)$-pickable permutation. It contributes $T_2(N-1, k_2, k_1)$.

\textbf{Case 3:} The last position has the value $N$. Then the value $N-1$ is at some position $i \in [k_1+1, N-1]$ since it must be picked as a left-to-right maximum. The positions $[1, i-1]$ form a non-$(k_2,k_1)$-pickable permutation and it contributes $T_2(i-1, k_2, k_1)$. We also have $B(i-1, N-i-1)$ counting the inversions between positions $[1, i-1]$ and $[i+1, N-1]$. There are no restrictions on positions in $[i+1, N-1]$ and thus they contributes $(P_{N-i-1})!$. The value $N-1$ at position $i$ contributes $\theta^{N-i-1}$ with respect to the positions $[i+1, N-1]$.
\end{proof}

It turns out that we can solve the recurrence relation in Lemma~\ref{w2k1k2}. Recall that for $k \ge 1$, we have $$W_2(N, k) = \theta \cdot T_2(N, k) - \theta^{2N-2k+1} \cdot (P_{N-2})! \cdot P_k \cdot P_{k-1} $$ $$= \theta \cdot (P_k)! \cdot \{(1+\theta^2 \cdot P_{N-2})(1+\theta^2 \cdot P_{N-3}) \cdots (1+\theta^2 \cdot P_k) \cdot (1+\theta^2 \cdot P_{k-1}) - \theta^{2N-2k} \cdot P_{N-2} \cdot P_{N-3} \cdots P_{k-1}\}.$$ For $k=0$, we have $$W_2(N, 0) = \theta \cdot T_2(N, 0).$$ 

\begin{lemma}
For $k_2 \le k_1 \le N-3$, $$W_2(N, k_2, k_1) = (\theta^{2N-2k_1+1} \cdot \frac{(P_{N-2})!}{(P_{k_1 - 2})!} + \theta^{2N-2k_1-1} \cdot \frac{(P_{N-2})!}{(P_{k_1 - 2})!} \cdot \sum\limits_{i=k_1-1}^{N-2} \frac{1}{P_i} $$
$$+ \theta^{N-k_1-2} \cdot \frac{(P_{N-2})!}{(P_{k_1 - 2})!} \cdot \sum\limits_{i=k_1-1}^{N-3} \frac{\theta^{i-k_1+1} \cdot (1 + \theta + \ldots + \theta^{N-i-3})}{P_i \cdot P_{i+1}}) \cdot T_2(k_1, k_2) - \theta^{2N-2k_2+1} \cdot (P_{N-2})! \cdot P_{k_2} \cdot P_{k_2-1}.$$

Moreover, when $N = k_1+1$, $$W_2(k_1+1, k_2, k_1) = \theta \cdot (1 + \theta^2 \cdot P_{k_1-1}) \cdot T_2(k_1, k_2) - \theta^{2k_1-2k_2+3} \cdot (P_{k_1-1})! \cdot P_{k_2} \cdot P_{k_2-1}; \text{ and when }N = k_1+2,$$
$$W_2(k_1+2, k_2, k_1) = (1 + \theta^3 \cdot P_{k_1-1} + \theta^3 \cdot P_{k_1} + \theta^5 \cdot P_{k_1} \cdot P_{k_1-1}) \cdot T(k_1, k_2) - \theta^{2k_1-2k_2+5} \cdot (P_{k_1})! \cdot P_{k_2} \cdot P_{k_2-1}.$$
\end{lemma}

\begin{proof}
For $N = k_1+1$, $(k_2, k_1)$-winnable permutations are the same as $k_2$-winnable permutations. Thus, 
$$W_2(k_1+1, k_2, k_1) = W_2(k_1+1, k_2) = \theta \cdot (1 + \theta^2 \cdot P_{k_1-1}) \cdot T_2(k_1, k_2) - \theta^{2k_1-2k_2+3} \cdot (P_{k_1-1})! \cdot P_{k_2} \cdot P_{k_2-1}.$$

Similarly to Theorem~\ref{solvew1}, we can solve for the case when $N \ge k_1+2$ by invoking the recurrence relation in Theorem~\ref{w2k1k2}. We omit the tedious simplification process.
\end{proof}

We are now ready to complete the proof of Theorem~\ref{resultk2k1}: Let $N \to \infty$. By~\eqref{t2-appendix} and note that $1+ \theta^2 \cdot P_k = P_{k+2} - \theta$, $$\frac{W_2(N, k_2, k_1)}{(P_N)!} = \frac{T_2(k_1, k_2)}{(P_{k_1-2})!} \cdot \frac{(1 - \theta)^2}{(1 - \theta^N) \cdot (1 - \theta^{N-1})} \cdot (\theta^{2N-2k_1+1} + \theta^{2N-2k_1-1} \cdot (1 - \theta) \cdot \sum\limits_{i = k_1-1}^{N-2} \frac{1}{1 - \theta^i}$$
$$+ \theta^{N-k_1-2} \cdot \sum\limits_{i = k_1-1}^{N-3} \frac{(\theta^{i-k_1+1} - \theta^{N-k_1-1})(1 - \theta)}{(1 - \theta^i)(1 - \theta^{i+1})}) - \theta^{2N-2k_2+1} \cdot \frac{(1 - \theta^{k_2})(1 - \theta^{k_2-1})}{(1 - \theta^N)(1 - \theta^{N-1})}$$
$$\to (1 - \frac{\theta(1 - \theta)}{1 - \theta^{k_2+1}}) \cdots (1 - \frac{\theta(1 - \theta)}{1 - \theta^{k_1-2}}) \cdot (1 - \theta^{k_1-1} - \theta + \theta^2) \cdot (1 - \theta^{k_1} - \theta + \theta^2) \cdot (\theta^{2N-2k_1+1} + \theta^{2N-2k_1-1} \cdot (1 - \theta) \cdot \sum\limits_{i = k_1-1}^{N-2} \frac{1}{1 - \theta^i}$$
\begin{equation}\label{winprobk1k2-2}
+ \theta^{N-k_1-2} \cdot \sum\limits_{i = k_1-1}^{N-3} \frac{(\theta^{i-k_1+1} - \theta^{N-k_1-1})(1 - \theta)}{(1 - \theta^i)(1 - \theta^{i+1})}) - \theta^{2N-2k_2+1} \cdot (1 - \theta^{k_2}) \cdot (1 - \theta^{k_2-1}).
\end{equation}

When $\frac{1}{2} < \theta < 1$, 
\begin{equation}\label{bounded}
(1 - \frac{\theta(1 - \theta)}{1 - \theta^{k_2+1}}) \cdots (1 - \frac{\theta(1 - \theta)}{1 - \theta^{k_1-2}}) \cdot (1 - \theta^{k_1-1} - \theta + \theta^2) \cdot (1 - \theta^{k_1} - \theta + \theta^2) \le 4.
\end{equation}

\textbf{Case 1:} $N-k_1 \to \infty$. Then $N - k_2 \to \infty$. By~\eqref{bounded}, $\frac{1}{2} < \theta < 1$, $\theta^{N-k_1} \text{ and } \theta^{N-k_2} \to 0$ exponentially, we have $~\eqref{winprobk1k2-2} \to 0$.

\textbf{Case 2:} $N - k_1 \not\to \infty$. Then $k_1 \to \infty$. We have to consider two subcases depending on whether $N - k_2 \to \infty$.

\textbf{Case 2.1:} $N - k_2 \to \infty$. Then $k_1 - k_2 \to \infty$ and $\theta^{2N-2k_2+1} \cdot (1 - \theta^{k_2}) \cdot (1 - \theta^{k_2-1}) \to 0$ exponentially. Again, we have to consider two subcases depending on whether $k_2 \to \infty$.

\textbf{Case 2.1.1:} $k_2 \to \infty$. Then since $0.75 < 1-\theta+\theta^2 < 1$ when $\frac{1}{2} < \theta < 1$,
$$(1 - \frac{\theta(1 - \theta)}{1 - \theta^{k_2+1}}) \cdots (1 - \frac{\theta(1 - \theta)}{1 - \theta^{k_1-2}}) \cdot (1 - \theta^{k_1-1} - \theta + \theta^2) \cdot (1 - \theta^{k_1} - \theta + \theta^2) \to (1-\theta+\theta^2)^{k_1-k_2} \to 0$$
exponentially. Thus, we have $\eqref{winprobk1k2-2} \to 0$.

\textbf{Case 2.1.2:} $k_2 \not\to \infty$. Then since $(1 - \frac{\theta(1 - \theta)}{1 - \theta^{k_1-2}}) \to 1 - \theta + \theta^2$, where $0.75< 1-\theta+\theta^2 < 1$, we have $$(1 - \frac{\theta(1 - \theta)}{1 - \theta^{k_2+1}}) \cdots (1 - \frac{\theta(1 - \theta)}{1 - \theta^{k_1-2}}) \le (1 - \frac{\theta(1 - \theta)}{1 - \theta^{k_1-2}})^{k_1-k_2-2} \to 0 \text{ and thus \eqref{winprobk1k2-2}} \to 0.$$  

\textbf{Case 2.2:} $N - k_2 \not\to \infty$. Then $k_2 \to \infty$ and  
$$\eqref{winprobk1k2-2} \to (1-\theta+\theta^2)^{k_1-k_2} \cdot (\theta^{2N-2k_1+1} + \theta^{2N-2k_1-1} \cdot (1 - \theta) \cdot (N - k_1) $$
$$ + \theta^{N-k_1-2} \cdot (1 - \theta^{N-k_1-1} - (N-k_1-1) \cdot \theta^{N-k_1-1} \cdot (1 - \theta)) - \theta^{2N-2k_2+1}.$$

Let $x = N - k_1$, $y = N-k_2$, and $$\eqref{winprobk1k2-2} \to (1-\theta+\theta^2)^{y-x} \cdot (\theta^{2x+1} + \theta^{2x-1} \cdot (1 - \theta) \cdot x + \theta^{x-2} \cdot (1-\theta^{x-1} - (x-1) \cdot \theta^{x-1} \cdot (1 - \theta))) - \theta^{2y+1} =:g(x, y).$$
The maximum value of $g(x,y)$ is positive and thus the optimal $(k_2, k_1)$-strategy must satisfy $N-k_1 \not \to \infty$ and $N-k_2 \not \to \infty$. \hfill \qed
\end{document}